\documentclass[11pt]{article}
\usepackage{mathrsfs}
\usepackage{amsmath,amssymb, amsthm, amsopn,
  amsfonts,epsfig,graphics,relsize,exscale}{}
\usepackage{hyperref}
\usepackage{enumerate}
\usepackage{pgf,tikz,pgfplots}
\hypersetup{
    colorlinks=true,
    linkcolor=red,
    filecolor=magenta,      
    urlcolor=cyan,
}
\usepackage[all]{xy}
\usepackage{graphicx}
\usepackage[T1]{fontenc}
\usepackage[latin1]{inputenc}
\usepackage{color,tocvsec2}
\usepackage{fouriernc}

\usepackage{pgf,tikz,pgfplots}
\pgfplotsset{compat=1.15}
\usepackage{mathrsfs}
\usetikzlibrary{arrows}
\usetikzlibrary{decorations.pathreplacing}
\definecolor{rd}{rgb}{1,0.3,0.35}
\newcommand{\red}{}

\usepackage[a4paper,width=14.66cm,top=2.52cm,bottom=2.52cm]{geometry}
\usepackage{makeidx}

\makeindex

\newcommand{\field}[1]{\mathbb{#1}} \newcommand{\rz}{\field{R}}
\newcommand{\cz}{\field{C}} \newcommand{\nz}{\field{N}}
\newcommand{\zz}{\field{Z}} \newcommand{\tz}{\field{T}}

\newcommand{\rank}{{\textrm{rank\,}}}

\DeclareMathOperator\supp{\textrm{supp}}

\newcommand{\ad}{{\text{ad\,}}}
\newcommand{\ccap}{\mathop{\cap}}
\newcommand{\ccup}{\mathop{\cup}}

\newtheorem{theorem}{Theorem}[section]
\newtheorem{lemma}[theorem]{Lemma}

\newtheorem{proposition}[theorem]{Proposition}
\newtheorem{definition}[theorem]{Definition}

\def\dim{\textrm{dim\;}}

\title{Mourre theory for analytically fibered operators revisited}
\author{
F.~Nier
\thanks{
LAGA, Universit{\'e} de Paris XIII, 99 avenue
J.B.~Cl{\'e}ment, F-93430~Villetaneuse, France. nier@math.univ-paris13.fr}\\
\red C.~G{\'e}rard
\thanks{Universit{\'e} Paris-Saclay, UMR-CNRS~8628, Laboratoire math{\'e}matiques d'Orsay, F-91405 Orsay, christian.gerard@universite-paris-saclay.fr }}
\begin{document}
\maketitle
\begin{abstract}
  About 25 years ago our article ``Mourre theory for analytically fibered operators'' was published in J.~of Functional Analysis. This article proposed a general construction of a conjugate operator for a wide class of self-adjoint analytically fibered hamiltonians, provided that one accepts a more accurate notion of threshold. It is only recently that Olivier Poisson mentionned us a problem with the statement that $H_{0}\in \mathcal{C}^{\infty}(A_{I})$\,. Actually even $H_{0}\in \mathcal{C}^{2}(A_{I})$ or $H_{0}\in \mathcal{C}^{1+0}(A_{I})$\,, which is crucial for the full application of Mourre theory, is problematic with our initial construction. However the statement and the construction can be modified in order to make work all the theory. This is explained here.
\end{abstract}
\\
\section{Introduction}
\label{sec:intro}

In \cite{GeNi} we introduced the notion of an analytically fibered operator
$$
H_{0}=\int^{\oplus}_{M}H_{0}(k)~dv(k)
$$
where:
\begin{enumerate}[(i)]
\item $M$ is a real analytic manifold and $dv(k)$ stands for a volume density which can be written $dv(k)=e^{2V(k)}~dk_{1}\ldots dk_{\dim M}$ in any coordinate chart with $V(k)$ real analytic;
\item $\mathcal{H}'$ is a Hilbert space, $\mathcal{L}(\mathcal{H}')$ and $\mathcal{L}^{\infty}(\mathcal{H}')$ respectively denote the spaces of bounded and compact operators in $\mathcal{H}'$\,;
\item for all $k\in M$\,, $(H_{0}(k),D(H_{0}(k)))$ is a self-adjoint operator in $\mathcal{H}'$ such that $M\ni k\mapsto (H_{0}(k)+i)^{-1}\in \mathcal{L}^{\infty}(\mathcal{H}')$ is real analytic (actually this can be replaced by $(H_{0}(k)+i)^{-1}$ is a real analytic section of $\mathcal{L}^{\infty}(\mathcal{F})$ where $\mathcal{F}$ is a real analytic Hilbert bundle $p_{\mathcal{F}}:\mathcal{F}\to M$ with fiber $\mathcal{F}_{k}\sim \mathcal{H}'$)\,;
\item when $\Sigma=\left\{(k,\lambda)\in M\times \rz,\lambda\in \mathrm{Spec}(H_{0}(k))\right\}$ is the analytic characteristic variety in $M\times\rz$\,, the map $p_{\Sigma\to\rz}:\Sigma\to \rz$ given by $p_{\Sigma\to \rz}(k,\lambda)=\lambda=p_{\rz}(k,\lambda)$ is proper. We distinguish by notations the natural projection $p_{\rz}:M\times \rz\to \rz$ from its restriction to $\Sigma$\,, $p_{\Sigma\to \rz}=p_{\rz}\big|_{\Sigma}$\,.
\end{enumerate}
After noticing that the analytic characteristic variety can be partitionned in the semi-analytic sets according to the multiplicity $\mu$ of the eigenvalue $\lambda$ of $H_{0}(k)$\,, $\Sigma=\bigsqcup_{\mu}\Sigma_{\mu}$\,, the general theory of subanalytic sets (see \cite{Hardt,Hiro,DeLe}) tells us that there exists a (locally finite) Whitney stratification $\mathcal{S}$ of $\Sigma$ compatible with $p_{\Sigma\to\rz}$\,, which means here:
\begin{itemize}
\item $\Sigma=\bigsqcup_{S\in \mathcal{S}}S$\,, where any $S\in \mathcal{S}$ is an analytic connected manifold {\red e}mbedded in $\rz\times M$;
\item $\overline{S}\cap S'\neq \emptyset$  implies $S'=S$ or $S'\subset \partial S\stackrel{def}=\overline{S}\setminus S$\,;
\item for all $S\in\mathcal{S}$ there exists $\mu_{S}$ such that $S\subset \Sigma_{\mu_{S}}$;
\item for all $S\in \mathcal{S}$ either $p_{\Sigma\to\mathrm{\rz}}(S)$ is an open interval and $dp_{\Sigma\to \rz}\big|_{TS}$ does not vanish  (rank of $dp_{\Sigma\to \rz}\big|_{TS}$ is constantly equal to $1$) or $p_{\Sigma\to \mathrm{\rz}}(S)=\left\{1pt\right\}$ (rank of $dp_{\Sigma\to\rz}\big|_{TS}$ is constantly equal to $0$)\,.
\end{itemize}
 We 
refer the reader to the founding articles \cite{Hardt}\cite{Hiro} and  to
\cite{Loja}\cite{BiMi} for a panoramic or historical presentation. 
Another reference concerned with the analysis of subanalytic Lipschitz functions is  \cite{DeLe}. The application of these general ideas to our case is given in \cite{GeNi}. 
We called the locally finite set $\tau=\bigsqcup_{\rank dp_{\Sigma\to \rz}\big|_{TS}=0}p_{\Sigma\to \rz}(S)$ in $\rz$\,, the set of \emph{thresholds} for $H_{0}$\,.
Our main statement in \cite{GeNi} was
\newtheorem{AncThm}{``Theorem''~3.1 in \cite{GeNi}}
\begin{AncThm}
  There exists a discrete set $\tau$ determined by $H_{0}$ so that for any interval $I\Subset \rz\setminus\tau$\,, there exists an operator $A_{I}$ essentially self-adjoint on $D(A_{I})=\mathcal{C}^{\infty}_{comp}(M;\mathcal{F})$ satisfying the following properties:
  \begin{description}
  \item[i)] For all $\chi\in \mathcal{C}^{\infty}_{comp}(I)$\,, there exists a constant $c_{\chi}>0$ so that
    \begin{equation}
      \label{eq:localMourre}
\chi(H_{0})[H_{0},iA_{I}]\chi(H_{0})\geq c_{\chi}\chi(H_{0})^{2}\,.
\end{equation}
\item[ii)] The multi-commutators $\mathrm{ad}_{A_{I}}^{k}(H_{0})$ are bounded for all $k\in \nz$\,.
\item[iii)] The operator $A_{I}$ is a first order differential operator in $k$ whose coefficients belong to $\mathcal{C}^{\infty}(M;\mathcal{L}(\mathcal{F}))$ and there exists $\chi\in \mathcal{C}^{\infty}_{comp}(\rz\setminus \tau)$ such that $A_{I}=\chi(H_{0})A_{I}=A_{I}\chi(H_{0})$\,.
  \end{description}
\end{AncThm}

As mentionned above, the second statement \textbf{ii)} is problematic with the construction of \cite{GeNi}. It was actually not carefully checked in \cite{GeNi}. However the general idea works after some modifications. Below is the proper statement.
\begin{theorem}
  \label{th:AI1I}
  For any interval $I\Subset \rz\setminus \tau$\,, there exists an operator $\tilde{A}_{I}$ essentially self-adjoint on $D(\tilde{A}_{I})=\mathcal{C}^{\infty}_{comp}(M;\mathcal{F})$ which satisfies the following properties:
  \begin{enumerate}
  \item[i)] For all $\lambda\in I$ there exists $\delta_{\lambda}>0$ such that
$$
1_{[\lambda-\delta_{\lambda},\lambda +\delta_{\lambda}]}(H_{0})[H_{0},i\tilde{A}_{I}]1_{[\lambda-\delta_{\lambda},\lambda+\delta_{\lambda}]}(H_{0})\geq \frac{1}{2}1_{[\lambda-\delta_{\lambda},\lambda+\delta_{\lambda}]}(H_{0})\,.
$$
\item[ii)] The multi-commutators $\mathrm{ad}_{\tilde{A}_{I}}^{k}(H_{0})$ are bounded for all $k\in \nz$ and $H_{0}\in \mathcal{C}^{\infty}(\tilde{A}_{I})$\,.
\item[iii)]  The operator $\tilde{A}_{I}$ is a  first order differential
{\red operator} in $k$  whose coefficients belong to $\mathcal{C}^{\infty}(M;\mathcal{L}(\mathcal{F}))$ and there exists $\chi\in \mathcal{C}^{\infty}_{comp}(\rz\setminus \tau)$  such that $\tilde{A}_{I}=\chi(H_{0})\tilde{A}_{I}=\tilde{A}_{I}\chi(H_{0})$\,.
  \end{enumerate}
\end{theorem}
The construction of the new conjugate operator $\tilde{A}_{I}$ slightly differs from the one of $A_{I}$ in order to ensure the condition~\textit{ii)}. The price to pay is about Mourre's inequality \textit{i)} which is now only local (with a uniform constant which is here $\frac{1}{2}$). But we recall that such a local inequality is sufficient for developping all the theory. Additional comments are given in Section~\ref{sec:comments} after the construction of $\tilde{A}_{I}$ is done.
\section{Preliminary remarks and notations}
\label{sec:prel}
The notion of vector (Hilbert) bundles although pointed out and later used in \cite{GeNi2} was not accurately enough considered in \cite{GeNi}. We specify its use now. The local analysis of \cite{GeNi}, putting the stress of locally scalar differential operators, was essentially correct. But gluing all the pieces together, where the dimensions of vector subbundles change, must be done carefully by taking into account the stratified struture of $\Sigma$\,. Elementary remarks and notations are collected here, while we recall the local initial construction of the operator $A_{I}$\,. Its correct variation $\tilde{A}_{I}$ will be done later.

\subsection{Open covering and partition of unity on $\Sigma$}
\label{sec:partunit}

The main idea of \cite{GeNi} which still works is that the construction and the analysis can be localized on the characteristic variety $\Sigma$\,. Because $p_{\Sigma\to \rz}:\Sigma\to \rz$ is proper, $p_{\Sigma\to\rz}^{-1}(\overline{I})\subset \Sigma$ is compact and  an open covering $\bigcup_{i\in \mathcal{I}}\Omega_{i}$  of $p_{\Sigma\to\rz}^{-1}(\overline{I})$ can always be assumed finite. A partition of unity $\sum_{i\in \mathcal{I}}g_{i}(\lambda,k)^{2}\equiv 1$ over a neighborhood of $p_{\Sigma\to \rz}^{-1}(\overline{I})$\,, subordinate to $\bigcup_{i\in \mathcal{I}}\Omega_{i}$\,, is made with $g_{i}\in \mathcal{C}^{\infty}_{comp}(\Omega_{i})$\,.\\
Like in \cite{GeNi}, the open sets $\Omega_{i}$\,, $i\in \mathcal{I}$\,, will be cylindrical above $M$\,, with $\Omega_{m}=\bigcup_{n=1}^{N_{m}}\omega_{m}\times J_{m,n}=\bigcup_{m=1}^{N_{m}}\Omega_{m,n}$\,,  $\omega_{m}$ open set of $M$ and $J_{m,n}$ open interval. The finite family $\mathcal{I}$ can then be either $\mathcal{I}=\left\{1,\ldots, m_{I}\right\}$ or $\mathcal{I}=L(I)\stackrel{def}{=}\left\{(m,n)\in\nz^{2},\, 1\leq n\leq N_{m},\, 1\leq m\leq m_{I}\right\}$ for a more local presentation on the energy axis. We use the convention that $\omega$ is an open set of $M$ and $\Omega$ is an open set in $M\times\rz$\,.\\
It is convenient to introduce another interval $\tilde{I}$ such that
$$
I\Subset \tilde{I}\Subset \rz\setminus\tau
$$
and all the open intervals $J_{m,n}$\,, $1\leq n\leq N_{m}$\,, $1\leq m\leq m_{I}$\,, will satisfy $J_{m,n}\Subset \tilde{I}$\,.
With $I$ we shall associate the finite  ($\mathcal{S}$ is locally finite and $p_{\Sigma\to\rz}^{-1}(\overline{I})$ is compact) set of strata
$$
\mathcal{S}_{I}=\left\{S\in \mathcal{S}\,,\; \overline{S}\cap p_{\Sigma\to \rz}^{-1}(\overline{I})\neq \emptyset\right\}\,.
$$
The assumption $\overline{I}\subset \rz\setminus \tau$ implies
$$
\forall S\in \mathcal{S}_{I}\,,\;\mathrm{rank}\,dp_{\Sigma\to \rz}\big|_{TS}\equiv 1\,.
$$
\begin{proposition}
  \label{pr:covering}
  Let $I$\,, $\tilde{I}$\,, be two open intervals such that $I\Subset \tilde{I}\Subset \rz\setminus \tau$\,. There exists open sets $\omega_{m}\Subset M$\,, for $1\leq m\leq M_{I}$\,, and open intervals $J_{m,n}\Subset \tilde{I}$ for $(m,n)\in L(I)=\left\{(m,n)\in \nz^{2}, 1\leq n\leq N_{m}\,, 1\leq m\leq M_{I}\right\}$ such that:
  \begin{enumerate}[(1)]
  \item setting $\Omega_{m,n}=\omega_{m}\times J_{m,n}$\,, $\bigcup_{(m,n)\in L(I)}\Omega_{m,n}$ is an open covering of $ p_{\Sigma\to \rz}^{-1}(\overline{I})$
    \,;
  \item $J_{m,n}\cap J_{m,n'}=\emptyset$ for $n\neq n'$\,, $1\leq n,n'\leq N_{m}$\,;
  \item for any $(m,n)\in L(I)$ there exists a unique stratum $S_{m,n}\in \mathcal{S}_{I}$ ($\rank dp_{\Sigma\to \rz}\big|_{TS_{m,n}}\equiv 1$) and $\Omega_{m,n}\cap \Sigma$ is a (finite) union of $\Omega_{m,n}\cap S'$\,, $S'\in \mathcal{S}_{I}$\,, $S_{m,n}$ included in every $\overline{S'}$ (i.e. $S'=S_{m,n}$ or $S_{m,n}\subset \partial S'=\overline{S'}\setminus S'$);
  \item $1_{J_{m,n}}(H_{0}(k))$ has a constant rank $\mu_{m,n}$ for $k\in \omega_{m}$ and there exists another interval $J'_{m,n}\Subset J_{m,n}$ such that $1_{J_{m,n}}(H_{0}(k))=1_{J'_{m,n}}(H_{0}(k))$ for all $k\in \omega_{m}$\,;
  \item for all $(m,n)\in L(I)$ there exists a real analytic section $W_{m,n}$ of $\mathcal{L}(\cz^{\mu_{m,n}};\mathcal{F}\big|_{\omega_{m}})$
    with $W_{m,n}(k):\cz^{\mu_{m,n}}\to \mathrm{Ran}\,1_{J_{m,n}}(H_{0}(k))$ unitary for all $k\in \omega_{m}$ so that
$$
W_{m,n}(k)^{*}H_{0}(k)1_{J_{m,n}}(H_{0}(k))W_{m,n}(k)=\tilde{H}_{m,n}(k)
$$
is a real analytic $\mu_{m,n}\times \mu_{m,n}$ matrix with respect to $k\in \omega_{m,n}$\,;
\item there exists a real analytic vector field $X_{m,n}$ on $\omega_{m,n}$ such that
  \begin{eqnarray*}
    \forall k\in \omega_{m,n}\,,&&
                                X_{m,n}\tilde{H}_{m,n}(k)\geq \frac{1}{2}\mathrm{Id}_{\cz^{\mu_{m,n}}}\,,\\
                             &&
                                [W_{m,n}\circ X_{m,n}\circ W_{m,n}^{*}]
                                H_{0}(k)1_{J_{m,n}}(H_{0}(k))\geq \frac{1}{2}1_{J_{m,n}}(H_{0}(k))\,;
  \end{eqnarray*}
\item if $\Omega_{m,n}\cap \Omega_{m',n'}\cap \Sigma \neq \emptyset$ and $\mu_{m,n}\geq \mu_{m',n'}$ then
  \begin{eqnarray*}
    &&
\forall k\in \omega_{m}\cap \omega_{m'}\,,\quad 1_{J_{m,n}}(H_{0}(k))1_{J_{m',n'}}(H_{0}(k))=1_{J_{m',n'}}(H_{0}(k))\,,
    \\
    \text{and}&&\left\{
                 \begin{array}[c]{ll}
                   &S_{m,n}\subset \partial S_{m',n'}\\
                   \text{or}& S_{m',n'}\subset \partial S_{m,n}\quad\text{and}\quad \mu_{m,n}=\mu_{m',n'}\\
                   \text{or}& S_{m',n'}=S_{m,n}\quad\text{and}\quad \mu_{m,n}=\mu_{m',n'}\,.
                 \end{array}
\right.
\end{eqnarray*}
\end{enumerate}
\end{proposition}
\begin{proof}
The construction of the open covering will be done in several steps. Actually because $K_{I}=p_{M}(p_{\Sigma\to\rz}^{-1}(\overline{I}))$ is compact in $M$ it suffices to consider carefully the situation in a neighborhood $\omega_{0}\Subset M$ of any $k_{0}\in p_{M}(p_{\Sigma\to\rz}^{-1}(\overline{I}))$\,. We use the index $m=0$ for this local analysis and in the different steps the open set $\omega_{0}\ni k_{0}$ will be reduced in order to fulfill all the properties. So let us fix $k_{0}\in p_{M}(p_{\Sigma\to\rz}^{-1}(\overline{I}))$\,. We start with three remarks: 
\begin{itemize}
\item The  set $\mathrm{Spec}(H_{0}(k_{0}))\cap \overline{I}$ is finite and is written $\left\{\lambda_{n}, 1\leq n \leq N_{0}\right\}$\,. 
\item With $m=0$, the properties 3 (uniqueness is imposed by $(k_{0},\lambda_{n})\subset S_{0,n}$),4,5,6 hold in any open neighborhood $\omega'\subset \omega_{0}$ of $k_{0}$\,. So these properties are hereditary after reducing the neighborhood $\omega_{0}$\,.
\item Locally the real analytic vector bundle $\mathcal{F}\big|_{\omega_{0}}$\,, for $\omega_{0}$ small enough, is transformed {\red by a} unitary transform $U_{0}(k):\mathcal{F}_{k}\to \mathcal{H}'$\,, real analytic with respect to $k\in \omega_{0}$\,, into the trivial bundle $\omega_{0}\times \mathcal{H}'$\,. So working with the trivial real analytic vector bundle $\omega_{0}\times \mathcal{H}'$ is not a restriction here. 
\end{itemize}
For a fixed $\lambda_{n}\in \overline{I}\cap \mathrm{Spec}(H_{0}(k_{0}))$\,, there is a unique stratum $S_{0,n}\in \mathcal{S}_{\overline{I}}$ and a unique $\mu_{0,n}\in \nz$ such that $(k_{0},\lambda_{n})\in S_{0,n}\subset \Sigma_{\mu_{0,n}}$\,. Consider the set $\mathcal{S}_{0,n}=\left\{S'\in \mathcal{S}\,, (k_{0},\lambda_{n})\in \overline{S'}\right\}\subset \mathcal{S}_{I}$\,.
Because the stratification $\mathcal{S}$ is locally finite{\red,} there exists an open neighborhood  $\omega^{1}_{0,n}\times J^{1}_{0,n}$ of $(k_{0},\lambda_{n})$ such that $\omega_{0,n}^{1}\times J^{1}_{0,n}\cap \overline{S'}\neq \emptyset$ is equivalent to
$S'\in \mathcal{S}_{0,n}$ (with $\mathcal{S}_{0,n}\subset \mathcal{S}_{I}$)\,.
Meanwhile $(k_{0},\lambda_{n})\in S_{0,n}$ gives $S_{0,n}\cap \overline{S'}\neq \emptyset$ and therefore $S_{0,n}=S'$ or $S_{0,n}\subset \partial S'$ with $\dim S_{0,n}<\dim S'$ in this case. This is the local in energy version of the property \textit{(3)} which is hereditary for another neihborhood $\omega_{0,n}^{2}\times J_{0,n}^{2}\subset \omega_{0,n}^{1}\times J_{0,n}^{1}$\,.
In particular we may assume $J_{0,n}^{2}\Subset \tilde{I}$ for every $1\leq n\leq N_{0}$ and $\overline{J_{0,n}^{2}}\cap\overline{J_{0,n'}^{2}}=\emptyset$ for $n\neq n'$\,, $1\leq n,n'\leq N_{0}$\,. From now the intervals $J_{0,n}=J_{0,n}^{2}$\,, $1\leq n\leq N_{0}$ are fixed and the properties \textit{(2)} and \textit{(3)} as well as $J_{0,n}\Subset \tilde{I}$ are granted.\\
Again for $n$ fixed there exists a path in $\gamma_{0,n}$ in $\cz$ around $J_{0,n}$ such that
$$
1_{\left\{\lambda_{n}\right\}}(H_{0}(k_{0}))=1_{J_{0,n}}(H_{0}(k_{0}))=\frac{1}{2i\pi}\int_{\gamma_{0,n}} \frac{dz}{z-H_{0}(k_{0})}\,.
$$
With the last formula the spectral projector
$$
\pi_{0,n}(k)=\frac{1}{2i\pi}\int_{\gamma_{0,n}}\frac{dz}{z-H_{0}(k)}
$$
is analytic in an open neighborhood $\omega_{0,n}^{2}\subset \omega_{0,n}^{1}$ of $k_{0}$\,.
Additionally choosing $\omega_{0,n}^{2}$ small enough implies $\|\pi_{0,n}(k)-\pi_{0,n}(k_{0})\|_{\mathcal{L}(\mathcal{H}')}<1$ and Nagy's formula (see e.g. \cite{MaSo}\cite{NeSo})
$$
W_{0,n}(k)=(1-(P_{2}-P_{1})^{2})^{-1/2}
\left[P_{2}P_{1}+(1-P_{2})(1-P_{1})\right]\quad\text{with}~P_{1}=\pi_{0,n}(k_{0})\,,\, P_{2}=\pi_{0,n}(k)\,,
$$
provides a unitary operator $W_{0,n}(k)$ in $\mathcal{H}'$\,, which depends analytically on $k\in \omega_{0,n}$ and such that $W_{0,n}(k)\pi_{0,n}(k_{0})=\pi_{0,n}(k)$\,. Therefore
$$
H_{0}(k)\pi_{0,n}(k)=\pi_{0,n}(k)H_{0}(k)\pi_{0,n}(k)=\frac{1}{2i\pi}\int_{\gamma_{0}}\frac{zdz}{z-H_{0}(k)}
$$
is unitarily equivalent to $\tilde{H}_{0,n}(k)=\pi_{0,n}(k_{0})W_{0,n}^{*}(k)H_{0}(k)W_{0,n}(k)\pi_{0,n}(k_{0})\in \mathcal{L}(\cz^{\mu_{0,n}})$ after identifying $\mathrm{Ran}\pi_{0,n}(k_{0})$ with $\cz^{\mu_{0,n}}$\,. Again this property, which is the condition \textit{(5)} is hereditary for any neighborhood $\omega_{0,n}^{3}\times J_{0,n}$ of $(k_{0},\lambda_{n})$ with $\omega_{0,n}^{3}\subset \omega_{0,n}^{2}$\,.\\
The local version of condition \textit{(4)} in $J_{0,n}$ is almost done: the rank of $1_{J_{0,n}}(H_{0}(k))$ is constantly equal to $\mu_{0,n}$ for $k\in \omega_{0,n}^{3}\subset \omega_{0,n}^{3}$\,. The function $\max \mathrm{Spec}(H_{0}(k))\cap J_{0,n}$ and $\min \mathrm{Spec}(H_{0}(k))$ are continuous functions of $k\in \omega_{0,n}^{3}$\,, equal to $\lambda_{0,n}$ when $k=k_{0}$ belonging to the open interval $J_{0,n}$\,. Hence the neighborhood $\omega^{3}_{0,n}$ of $k_{0}$ can be chosen small enough so that $1_{J_{0,n}}(H_{0}(k))=1_{J'_{0,n}}(H_{0}(k))$ with $J_{0,n}'\Subset J_{0,n}$ and this will be true for {\red any} neighborhood of $k_{0}$ $\omega_{0,n}^{4}\subset \omega_{0,n}^{3}$\,.\\
The local problem is thus reduced to the case of a self-adjoint  matrix $\tilde{H}_{0,n}(k)=\tilde{H}_{0,n}(k)^{*}\in M_{\mu_{0,n}}(\cz)$ which is real analytic with respect to $k\in \omega_{0,n}^{4}$\,.\\
Because $\mathrm{Tr}[\tilde{H}_{0,n}(k)]=\mathrm{Tr}\left[H_{0}(k)1_{J_{0,n}}(H_{0}(k))\right]$ is analytic on $\omega_{0,n}^{4}$\,, the set
$$
\tilde{S}_{0,n}=\left\{(k,\frac{1}{\mu_{0,n}}\mathrm{Tr}[\tilde{H}_{0,n}(k)]), k\in \omega_{0,n}^{4}\right\}
$$
is an analytic submanifold of $\omega_{0,n}^{4}\times J_{0,n}$ diffeomorphic to $\omega_{0,n}^{4}$ via the projection $p_{M}:\omega_{0,n}^{4}\times J_{0,n}\to \omega_{0,n}^{4}$\,. Since $S_{0,n}\cap (\omega_{0,n}^{4}\times J_{0,n})$ is a real analytic manifold contained in $\tilde{S}_{0,n}$\,, $p_{M}(S_{0,n}\cap (\omega_{0,n}^{4}\times J_{0,n}))$ is a real analytic submanifold of $\omega_{0,n}^{4}$\,. For $\omega_{0,n}^{4}$ small enough, there exists real analytic coordinates $(k',k'')\in\rz^{\dim S_{0,n}}\times \rz^{\dim M-\dim S_{0,n}}$ around $k_{0}$ such that $S_{0,n}\cap (\omega_{0,n}^{4}\times J_{0,n})=\left\{((k',0),\frac{1}{\mu_{0,n}}\mathrm{Tr}[\tilde{H}_{0,n}(k',0)])\,, (k',0)\in \omega_{0,n}^{4})\right\}$\,. Our assumption that $\rank dp_{\Sigma\to\rz}\big|_{TS_{0,n}}=1$ implies that the image of $\frac{\partial}{\partial \lambda}$ by  the projection  from $T(M\times\rz)\big|_{S_{0,n}}$ to $TS_{0,n}$
is a non vanishing real analytic vector field $X_{0,n}=\sum_{j=1}^{\dim S_{0,n}}a_{j}(k',0)\frac{\partial}{\partial k_{j}}$\,, with $k'=(k_{1},\ldots,k_{\dim S_{0,n}})$\,, such that $X_{0,n}\tilde{H}_{0,n}(k',0)=\mathrm{Id}_{\cz^{\mu_{0,n}}}$\,. Because $(k',k'')$ are coordinates on $\omega_{0,n}^{4}$\,, this vector field is a real analytic vector field on $\omega_{0,n}^{4}$ and $X_{0,n}\tilde{H}_{0,n}(k)$ is a real analytic self-adjoint matrix. By continuity the open set $\omega_{0,n}^{4}$ can be reduced such that $X_{0,n}\tilde{H}_{0,n}(k)\geq \frac{1}{2}\mathrm{Id}_{\cz^{\mu_{0,n}}}$\,. This proves the condition~6 and the  vector field $X_{0,n}$ can be taken unchanged, except the restriction, for any neighborhood of $k_{0}$ contained in $\omega_{0,n}^{4}$ \,.\\
Finally for any open neighborhood 
$$
\omega_{k_{0}}\subset \bigcap_{n=1}^{N_{0}}\omega_{0,n}^{4}
$$
of $k_{0}$\,  the inclusion $J_{0,n}\Subset \tilde{I}$ and the  conditions \textit{(2),(3),(4),(5),(6)} are satisfied by the family $$
(\Omega_{0,n}=\omega_{0}\times J_{0,n}, S_{0,n}, \mu_{0,n},W_{0,n}, \tilde{H}_{0,n}, X_{0,n})_{1\leq n\leq N_{0}}
\quad\text{with}\quad\omega_{0}=\omega_{k_{0}}\,.
$$
Without condition~\textit{(7)}, we could conclude with \textit{(1)} by extracting a finite covering out of $K_{I}\subset \bigcup_{k_{0}\in K_{I}} \omega_{k_{0}}$ of the compact set $K_{I}=p_{M}(p_{\Sigma\to \rz}^{-1}(\overline{I}))$\,.\\
The condition~\textit{(7)} requires an additional specification of the open neighborhood $\omega_{k_{0}}$ of $k_{0}$\,. It is a global property of the extracted finite covering $p_{\Sigma\to \rz}^{-1}(\overline{I})\subset \bigcup_{(m,n)\in L(I)}\Omega_{m,n}$ because it is about all the possible intersections\,. It is actually obtained by a  connectedness argument after choosing a specific family of local connected neighborhoods $(\omega_{k_{0}})_{k_{0}\in K_{I}}$  with all intersections connected. This is possible on a manifold.  {\red Let us fix any} riemannian metric $\gamma$ on $M$\,. {\red Then:}
\begin{itemize}
\item any geodesically convex set is arcwise connected and the intersection of geodesically convex sets is geodesically convex;
\item for any $k_{0}\in M$ there exists $\varepsilon_{\gamma,k_{0}}>0$ such that for all $\varepsilon\leq \varepsilon_{\gamma,k_{0}}$ the geodesic ball $B_{\gamma}(k_{0},\varepsilon)$ is geodesically convex (see \cite{Lee}-Exercise~6-4) 
\end{itemize}
Take for any $k_{0}\in K_{I}{\red =p_{M}(p_{\Sigma\to \rz}^{-1}(\overline{I}))}$
\begin{eqnarray*}
&&\omega_{k_{0}}=B_{\gamma}(k_{0},\varepsilon_{0})\quad,\quad
  2\varepsilon_{0}<\varepsilon_{\gamma,k_{0}}\\
  &&
B_{\gamma}(k_{0},2\varepsilon_{0})\subset \bigcap_{n=1}^{N_{0}}\omega_{0,n}^{4}
\end{eqnarray*}
and let $K_{I}\subset \bigcup_{m=1}^{m_{I}}\omega_{m}$ be the finite extracted covering with $\omega_{m}=B_{\gamma}(k_{m},\varepsilon_{m})$\,. The covering \textit{(1)} is given by $\Omega_{m,n}=\omega_{m}\times J_{m,n}$\,. And $\Omega_{m,n}\cap \Omega_{m',n'}\cap \Sigma\neq \emptyset$ implies $B_{\gamma}(k_{m},\varepsilon_{m})\cap B_{\gamma}(k_{m'},\varepsilon_{m'})\neq \emptyset$ and therefore $d_{\gamma}(k_{m},k_{m'})< \varepsilon_{m}+\varepsilon_{m'}\leq 2\max(\varepsilon_{m},\varepsilon_{m'})$\,. By symmetry we can focus on the case $\max(\varepsilon_{m},\varepsilon_{m'})=\varepsilon_{m}$ for which $k_{m'}\in B_{\gamma}(k_{m},2\varepsilon_{m})\subset \omega_{m}^{4}$\,. On the connected set $B_{\gamma}(k_{m},2\varepsilon_{m})\cap \omega_{m'}\subset \omega_{m}^{4}\cap \omega_{m'}^{4}$\,, the spectral projector $1_{J_{m,n}}(H_{0}(k))1_{J_{m',n'}}(H_{0}(k))$ is real analytic with respect to $k$ with a constant rank $\leq \min (\mu_{m,n},\mu_{m',n'})$\,. Because this set contains $k_{m'}$ for which $J_{m,n}\cap J_{m',n'}\cap \mathrm{Spec}(H_{0}(k_{m'}))=\left\{\lambda_{m',n'}\right\}$ with the multiplicity $\mu_{m',n'}$\,,
$$
\mu_{m',n'}\leq \mu_{m,n} \quad\text{and}\quad 1_{J_{m,n}}(H_{0}(k))1_{J_{m',n'}}(H_{0}(k))=1_{J_{m',n'}}(H_{0}(k))
$$
must hold for all $k\in B_{\gamma}(k_{m},2\varepsilon_{m})\cap B_{\gamma}(k_{m'},\varepsilon_{m'})\supset \omega_{m}\cap \omega_{m'}$\,. Because in this case
$(k_{m'},\lambda_{m',n'})\in S_{m',n'}$ we get
$$
S_{m',n'}\cap \omega_{m,n}^{4}\times J_{m,n}\supset
S_{m',n'}\cap B(k_{m},2\varepsilon_{m})\times J_{m,n}\supset \left\{(k_{m'},\lambda_{m',n'})\right\}\neq \emptyset
$$
and the property \textit{(3)} implies $S_{m,n}\subset\partial S_{m',n'}$ or ($S_{m',n'}=S_{m,n}$ and $\mu_{m,n}=\mu_{m',n'}$). By symmetry with $\varepsilon_{m'}=\max(\varepsilon_{m},\varepsilon_{m'})$, the equality $\mu_{m,n}=\mu_{m',n'}$ is compatible with $S_{m',n'}\in \partial S_{m,n}$\,. Although the ordering of $\varepsilon_{m}$ and $\varepsilon_{m'}$ does not appear in the final statement, the condition \textit{(3)} says that $\varepsilon_{m}=\varepsilon_{m'}$ is possible only when $S_{m,n}=S_{m',n'}$\,.
\end{proof}

Once the open covering $p_{\Sigma\to\rz}^{-1}(\overline{I})\Subset \bigcup_{(m,n)\in L(I)}\Omega_{m,n}$ is chosen according to Proposition~\ref{pr:covering} with $\Omega_{m,n}=\omega_{m}\times J_{m,n}$\,, several objects can be associated with it.\\
Take a partition of unity
\begin{equation}
  \label{eq:partunitM}
\sum_{m=1}^{M_{I}}g_{m}^{2}(k)\equiv 1~\text{on a neighborhood of}~K_{I}=p_{M}(p_{\Sigma\to\rz}^{-1}(\overline{I}))\,.
\end{equation}
Because $\mathrm{Spec}(H_{0}(k))\cap \overline{I}\subset \bigsqcup_{n=1}^{N_{m}}J_{m,n}$ for $k\in \omega_{m}$ and all $m\in \left\{1,\ldots,m_{I}\right\}$\,, we obtain:
$$
\sum_{(m,n)\in L(I)}g_{m}^{2}(k)1_{J_{m,n}}(\lambda)\equiv 1~\text{on a neighborhood of}~p_{\Sigma\to\rz}^{-1}(\overline{I})\,.
$$
With $(m,n)\in L(I)$ we associated the interval $J_{m,n}'\Subset J_{m,n}$ in Proposition~\ref{pr:covering}-\textit{(4)} such that $1_{J_{m,n}}(H_{0}(k))=1_{J_{m,n}'}(H_{0}(k))$ for all $k\in \omega_{m}$\,.
By taking $\chi_{m,n}\in \mathcal{C}^{\infty}_{comp}(J_{m,n})$ with $\chi_{m,n}\equiv 1$ on a neighborhood of $\overline{J}_{m,n}'$\,, we obtain the smooth partition of unity
\begin{eqnarray}
  \label{eq:partunitSig}
&&\sum_{(m,n)\in L(I)}g_{m}(k)^{2}\chi_{m,n}^{2}(\lambda)\equiv 1
~\text{on a neighborhood of}~p_{\Sigma\to\rz}^{-1}(\overline{I})\,,
  \\
\label{eq:chi1J}\text{with}&&
                \forall k\in \omega_{m}\,, \quad g_{m}(k)1_{J_{m,n}}(H_{0}(k))=g_{m}(k)\chi_{m,n}(H_{0}(k))\,.
\end{eqnarray}
\begin{definition}
  \label{de:asym}
  With the open covering of Proposition~\ref{pr:covering}, 
the set of  pairs $(m,n)\in L(I)$ is naturally endowed with an {\red asymmetric} relation $\vartriangleleft$ ({\red asymmetric} means that $(x\vartriangleleft y~\mathrm{and}~y\vartriangleleft x)$ never happens). Actually $(m,n)\vartriangleleft (m',n')$ is defined by:
\begin{enumerate}[(i)]
\item $(\omega_{m}\times J_{m,n})\cap ( \omega_{m'}\times J_{m',n})\cap \Sigma\neq \emptyset$\,;
\item $\dim S_{m,n}< \dim S_{m',n'}$ \,. 
\end{enumerate}
\end{definition}
\begin{lemma}
  \label{le:lefttri}
  For $(m,n),(m',n')\in L(I)$ the relation $(m,n)\vartriangleleft (m',n')$ implies
  \begin{enumerate}[(1)]
  \item  $m\neq m'$ and $S_{m,n}\subset \partial S_{m',n'}$\,;
  \item for all $k\in \omega_{m}\cap \omega_{m'}$\,, all the operators $\chi_{m,n}(H_{0}(k))\chi_{m',n'}(H_{0}(k))$\,, $1_{J_{m,n}}(H_{0}(k))1_{J_{m',n'}}(H_{0}(k))$\,, $1_{J_{m',n'}}(H_{0}(k))$ and $\chi_{m',n'}(H_{0}(k))$ are {\red equal.}
  \end{enumerate}
  The result \textit{(2)} still holds when the assumption $(m,n)\vartriangleleft (m',n')$ is replaced by the more general condition ($\Omega_{m,n}\cap \Omega_{m',n'}\cap \Sigma\neq \emptyset$ and $\overline{S_{m,n}}\subset \overline{S_{m',n'}}$)\,.
\end{lemma}
\begin{proof}
  If $m=m'$ with $\Omega_{m,n}\cap \Omega_{m,n'}\neq\emptyset$ then $J_{m,n}\cap J_{m,n'}\neq \emptyset$\,, where Proposition~\ref{pr:covering}-\textit{(2)} implies $n=n'$ and Proposition~\ref{pr:covering}-\textit{(3)} implies $S_{m,n}=S_{m',n'}$ in contradiction with Definition~\ref{de:asym}-\textit{ii)}. In the second statement the equality $1_{J_{m,n}}(H_{0}(k))1_{J_{m',n'}}(H_{0}(k))=1_{J_{m',n'}}(H_{0}(k))$ is given by the classification of Proposition~\ref{pr:covering}-\textit{(7)} where $\dim S_{m,n}<\dim S_{m',n'}$ {\red eliminates} the cases $\mu_{m,n}< \mu_{m',n'}$ and $S_{m,n}=S_{m',n'}$\,. The choice of the cut-off $\chi_{m,n}$ was done such that $\chi_{m,n}(H_{0}(k))=1_{J_{m,n}}(H_{0}(k))$ for all $k\in \omega_{m}$ and all $(m,n)\in L(I)$\,.\\
The more general condition includes the additional case $S_{m,n}=S_{m',n'}$ where all the operators of \textit{(2)} are equal to $1_{J_{m,n}}(H_{0}(k))$\,.
\end{proof}
The classification of Proposition~\ref{pr:covering}-\textit{(7)} says that
$\Omega_{m,n}\cap \Omega_{m',n'}\cap \Sigma\neq \emptyset$ is classified into the three {\red mutually exclusive} cases
$$
S_{m,n}\vartriangleleft S_{m',n'}\quad\text{or}\quad S_{m',n'}\vartriangleleft S_{m,n}\quad \text{or}\quad S_{m,n}=S_{m',n'},
$$
while the last case does not mean $(m,n)=(m',n')$\,.
It can be summarized as
$$
(\Omega_{m,n}\cap \Omega_{m',n'}\cap\Sigma \neq\emptyset)\Rightarrow
(\overline{S_{m,n}}\subset \overline{S_{m',n'}}~\text{or}~\overline{S_{m',n'}}
\subset\overline{S_{m,n}})
$$
while we recall that $\overline{S}\subset\overline{S'}$\,, which means $S=S'$ or $S\subset \partial S'=\overline{S'}\setminus S'$\,, is an order relation on $\mathcal{S}$\,.
This leads to the following {\red useful} result.
\begin{proposition}
  \label{pr:order} Any family $(m_{\ell},n_{\ell})_{1\leq \ell\leq L}$ of $L(I)$ such that $\bigcap_{\ell=1}^{L}\Omega_{m_{\ell},n_{\ell}}\cap \Sigma\neq \emptyset$ can be ordered such that $\overline{S_{m_{\ell},n_{\ell}}}\subset \overline{S_{m_{\ell+1},n_{\ell+1}}}$\,.\\
  With such an order, all the operators  $\chi_{m_{\ell},n_{\ell}}(H_{0}(k))\chi_{m_{\ell'},n_{\ell'}}(H_{0}(k))$\,, $1_{J_{m_{\ell},n_{\ell}}}(H_{0}(k))1_{J_{m_{\ell'},n_{\ell'}}}(H_{0}(k))$\,, $1_{J_{m_{\ell'},n_{\ell'}}}(H_{0}(k))$ and $\chi_{m_{\ell'},n_{\ell'}}(H_{0}(k))$ are {\red equal} for $k\in \omega_{m_{\ell}}\cap \omega_{m_{\ell'}}$ when $1\leq \ell\leq \ell'\leq L$\,.
  In particular all the operators  $\chi_{m_{\ell},n_{\ell}}(H_{0}(k))\chi_{m_{L},n_{L}}(H_{0}(k))$\,, $1_{J_{m_{\ell},n_{\ell}}}(H_{0}(k))1_{J_{m_{L},n_{L}}}(H_{0}(k))$\,, $1_{J_{m_{L},n_{L}}}(H_{0}(k))$ and $\chi_{m_{L},n_{L}}(H_{0}(k))$ are {\red equal} for $k\in \bigcap_{\ell'=1}^{L} \omega_{m_{\ell'}}$ for all $\ell\in \left\{1,\ldots, L\right\}$\,.
\end{proposition}

\subsection{Unitary connections}
\label{sec:connections}
A connection $\nabla$\,, on a Hilbert bundle $p_{\mathcal{F}}:\mathcal{F}\to  M$ is a $\rz$ (or $\cz$)-linear map from  $\mathcal{C}^{\infty}(M;\mathcal{F})$ to $\mathcal{C}^{\infty}(M;T^{*}M\otimes \mathcal{F})$ with the {\red Leibniz} rule
$$
\forall f\in \mathcal{C}^{\infty}(M;\rz),\forall s\in \mathcal{C}^{\infty}(M;\mathcal{F})\,,\quad \nabla (fs)=(df) s+f(\nabla s)
$$
or in terms of covariant derivatives 
$$
\forall X\in \mathcal{C}^{\infty}(M;TX),\, \forall f\in \mathcal{C}^{\infty}(M;\rz),\forall s\in \mathcal{C}^{\infty}(M;\mathcal{F})\,,\quad \nabla_{X} (fs)=(Xf) s+f(\nabla_{X})s\,.
$$
The curvature $R^{\nabla}\in $ is given by
$$
R^{\nabla}(X,Y)=\nabla_{X}\nabla_{Y}-\nabla_{Y}\nabla_{X}-\nabla_{[X,Y]}
$$
and it defines an element of $\mathcal{C}^{\infty}(M;\Lambda^{2}T^{*}M\otimes \mathrm{End}(\mathcal{F}))$\,, where in all our cases $\dim \mathcal{F}_{k}<\infty$ or $\dim \mathcal{F}_{k}=\infty$\,, $\mathrm{End}(\mathcal{F})=\mathcal{L}(\mathcal{F})$\,.\\
In a local chart open set $\omega\subset M$\,, there exists a unitary map $U(k)$ in $\mathcal{H}'$ with depends analytically on $k\in \omega$ such that
$$
u(k)\mapsto e^{-V(k)}U(k)u(k)
$$
defines a unitary map from $L^{2}(\omega,dk;\mathcal{H}')$ to $\mathcal{F}\big|_{\omega}$ endowed with the metric $\langle~,~\rangle_{\mathcal{H}'}e^{2V(k)}$\,.
Hence on $\mathcal{F}\big|_{\omega}$ there always exists a flat unitary connection given by
\begin{equation}
  \label{eq:connectriv}
\nabla^{\omega}=e^{-V(k)}U(k)[\sum_{i=1}^{\dim M}\frac{\partial}{\partial k_{i}}{dk_{i}}]e^{V(k)}U(k)^{-1}=U(k)\left[\sum_{i=1}^{\dim M}\frac{\partial}{\partial k_{i}}dk_{i}\right]U(k)^{-1}+dV\otimes \mathrm{Id}_{\mathcal{F}}
\end{equation}
Here flat means $R^{\nabla^{\omega}}\equiv 0$ while the unitarity of the Hilbert bundle with fiber $\mathcal{H}'$ and the metric $\langle~,~\rangle_{\mathcal{H}'}e^{2V(k)}$ means
$$
X\left[\langle s\,,\,s'\rangle_{\mathcal{H}'}e^{2V(k)}\right]=\langle \nabla^{\omega}_{X}s\,,\, s'\rangle_{\mathcal{H}'}e^{2V(k)}+\langle s\,,\,\nabla^{\omega}_{X}s'\rangle e^{2V(k)}\,.
$$
In particular when $X=\sum_{i=1}^{\dim M}X_{i}\partial_{k_{i}}\in \mathcal{C}^{\infty}(\omega;TM)$\,, $\nabla^{\omega}_{X}+\frac{1}{2}\mathrm{div}\,X$\,, with $\mathrm{div}\,X=\sum_{i=1}^{\dim M}\partial_{k_{i}}X_{i}$\,, is an antisymmetric operator on $\mathcal{C}^{\infty}_{comp}(\omega;\mathcal{F})$\,.\\
Let $M=\bigcup_{i\in \mathcal{I}}\omega_{i}$ be a locally finite open covering of $M$\,, with $\overline{\omega}_{i}$ compact, and let $\sum_{i\in \mathcal{I}}\chi_{i}^{2}$\,, $\chi_{i}\in \mathcal{C}^{\infty}_{comp}(\omega_{i})$ be a subordinate partition of unity with $\sum_{i\in \mathcal{I}}\chi_{i}^{2}\equiv 1$\,.  When $(\nabla_{i})_{i\in \mathcal{I}}$ is a family of unitary connections on $\mathcal{F}\big|_{\omega_{i}}$ then
\begin{equation}
  \label{eq:nablaM}
\nabla^{M}=\sum_{i\in \mathcal{I}}\chi_{i}\nabla_{i}\chi_{i}
\end{equation}
is a unitary connection on $\mathcal{F}$\,.\\
A connection $\nabla^{M}$ on $p_{\mathcal{F}}:\mathcal{F}\to M$\,, allows to define a connection still denoted by $\nabla^{M}$ on $p_{\mathcal{L}(\mathcal{F})}:\mathcal{L}(\mathcal{F})\to M$ via
$$
(\nabla^{M}_{X}A)s=\nabla^{M}_{X}(As)-A(\nabla^{M}_{X}s)
$$
which ensures the Leibnitz rule $\nabla^{M}(As)=(\nabla^{M} A)s+A(\nabla^{M} s)$\,.
The unitarity of $\nabla^{M}$ implies $(\nabla^{M}_{X}A)^{*}=-\nabla^{M}_{X}(A^{*})+\lambda_{X}A^{*}$\, {\red where the factor $\lambda_{X}A^{*}$ is due to the differentiation of the  volume form $dv(k)$ in the integration by part.}\\
Remember also that the difference between two connections on $p_{\mathcal{F}}:\mathcal{F}\to M$ is given by
$\nabla^{2}-\nabla^{1}=L\in \mathcal{C}^{\infty}(M;T^{*}M\otimes \mathcal{L}(\mathcal{F}))$\,. This holds true when $\mathcal{F}$ is finite dimensional and when $\dim \mathcal{F}=\infty$ the considered connections $\nabla^{2}$ are the ones for which the property holds for a given fixed connection $\nabla^{1}$\,, e.g. $\nabla^{1}=\nabla^{M}=\sum_{i\in \mathcal{I}}\chi_{i}\nabla^{i}\chi_{i}$\,. A typical example is when $\nabla^{2}$ is given by the same formula as $\nabla^{1}=\nabla^{M}$ but with a different partition of unity $\sum_{i\in \mathcal{I}}\chi_{2,i}^{2}\equiv 1$ subordinate to $M=\bigcup_{i\in \mathcal{I}}\omega_{i}$\,. The induced connections on $p_{\mathcal{L}(\mathcal{F})}:\mathcal{L}(\mathcal{F})\to M$ satisfy  $(\nabla^{2}_{X}-\nabla^{1}_{X})A=[L_{X},A]$\,. In particular, the factor $e^{2V(k)}$ in the metric $\langle~,~\rangle_{\mathcal{H}'}e^{2V(k)}$ can be forgotten while considering the connection $\nabla^{\omega}$ induced on $p_{\mathcal{L}(\mathcal{F})}:\mathcal{L}(\mathcal{F})\to \omega$ with $\nabla^{\omega}$ given by
\eqref{eq:connectriv} because $L_{X}=(dV(X))\mathrm{Id}_{\mathcal{H}'}$\,.  When both $\nabla^{2}$ and $\nabla^{1}$ are unitary connections, $L_{X}^{*}=(\nabla^{2}_{X}-\nabla_{X}^{1})^{*}=-(\nabla^{2}_{X}-\nabla^{1}_{X})=-L_{X}$ and $L_{X}$ is anti-symmetric.\\
In a local coordinate system $(k_{1},\ldots,k_{d})$ in $\omega\subset M$\,, the following relations
\begin{eqnarray}
  \label{eq:curvkikj}
  && \left[\nabla_{\frac{\partial}{\partial k_{i}}},\nabla_{\frac{\partial}{\partial k_{j}}}\right]=R^{\nabla}(\frac{\partial}{\partial k_{i}},\frac{\partial}{\partial k_{j}})\in \mathcal{C}^{\infty}(\omega;\mathcal{L}(\mathcal{F}))\,,\\
  \label{eq:conn21ki}
  && \nabla^{2}_{\frac{\partial}{\partial k_{i}}}-\nabla^{1}_{\frac{\partial}{\partial k_{i}}}\in \mathcal{C}^{\infty}(\omega;\mathcal{L}(\mathcal{F}))\,,
\end{eqnarray}
ensure that the following objects are well defined.
\begin{definition}
  \label{de:firstorder}
  Let $\omega$ be an open set of $M$\,.
  A differential operator $B$ of order $\mu$ with coefficients in $\mathcal{C}^{\infty}(\omega;\mathcal{L}(\mathcal{F}))$ is a local operator which can be written in any local coordinate system
  \begin{equation}
    \label{eq:diffopmu}
B=\sum_{|\alpha|\leq \mu}B_{\alpha}(k)\nabla_{\frac{\partial}{\partial k_{1}}}^{\alpha_{1}}\cdots \nabla_{\frac{\partial}{\partial k_{d}}}^{\alpha_{d}}\,.
\end{equation}
It is initially defined with the domain $D(B)=\mathcal{C}^{\infty}_{comp}(\omega;\mathcal{F})$ and it is said symmetric according to the general definition
$$
\forall \varphi,\psi\in \mathcal{C}^{\infty}_{comp}(\omega;\mathcal{F}),\quad
\langle \varphi\,,\, B\psi \rangle_{L^{2}(M;\mathcal{F})}=\langle B\varphi\,,\,\psi\rangle_{L^{2}(M;\mathcal{F})}\,.
$$
A first order differential operator $B$ with $\mathcal{C}^{\infty}(\omega;\mathcal{L}(\mathcal{F}))$ coefficients, is said to have a scalar principal part if there exists $X\in \mathcal{C}^{\infty}(\omega;TM)$ and $B_{0}\in \mathcal{C}^{\infty}(\omega;\mathcal{L}(\mathcal{F}))$ such that
\begin{equation}
  \label{eq:1stscalprinc}
B=\nabla_{X}+B_{0}\,.
\end{equation}
\end{definition}
In \eqref{eq:diffopmu} the covariant derivatives do not commute but \eqref{eq:curvkikj} ensures that the form \eqref{eq:diffopmu} is preserved, after modifying the coefficients $B_{\alpha}$\,, when the order is changed. The relation \eqref{eq:conn21ki} ensures that the general forms \eqref{eq:diffopmu} and \eqref{eq:1stscalprinc} are not changed when the connection $\nabla$ is changed.\\
We will focus on first order differential operators with $\mathcal{C}^{\infty}(\omega;\mathcal{L}(\mathcal{F}))$ coefficients where $\omega$ and the fiber bundle will be specified further.\\
When $\omega$ is a local chart open set with $\mathcal{F}\big|_{\omega}=e^{-V(.)}U(.)L^{2}(\omega;\mathcal{H}')$ and $\nabla^{\omega}$ defined by \eqref{eq:connectriv}, a first order differential operator can be written 
$$
B=e^{-V(k)}U(k)\left[\sum_{i=1}^{d}\underbrace{B_{1,i}(k)X_{i}(k)}_{=\tilde{B}_{1,i}(k)}\frac{\partial}{\partial_{k_{i}}}+B_{0}(k)\right]e^{V(k)}U(k)^{-1}
$$
with $B_{1,i},B_{0},\tilde{B}_{1,i}\in \mathcal{C}^{\infty}(\omega;\mathcal{L}(\mathcal{H}'))$ and $X=\sum_{i=1}^{k}X_{i}(k)\frac{\partial}{\partial_{k_{i}}}\in \mathcal{C}^{\infty}(\omega;TM)$\,.  It has a scalar principal part when all the $B_{1i}$'s equal $\mathrm{Id}_{\mathcal{H}'}$\,.\\
The compositions $BA$ or $AB$ with $B=B_{1}\nabla^{\omega}_{X}+B_{0}$\,, $B_{1},B_{0}\in \mathcal{C}^{\infty}(\omega;\mathcal{L}(\mathcal{F}))$ and $A\in \mathcal{C}^{\infty}(\omega;\mathcal{L})$ equal
$$
BA=B_{1}A\nabla^{\omega}_{X}+(B_{1}(\nabla^{\omega}_{X}A))+B_{0}A\quad\text{and}\quad
AB=AB_{1}\nabla^{\omega}_{X}+AB_{0}\,,
$$
and the commutator $[B,A]$ belongs to $\mathcal{C}^{\infty}(\omega;\mathcal{L}(\mathcal{F}))$ when $[B_{1},A]=0$\,, which is the case if $B$ has a scalar principal part.\\
When $[B^{1}_{1},B^{2}_{1}]=0$\,, the commutator (on $\mathcal{C}^{\infty}(\omega;\mathcal{F})$) $[B^{1},B^{2}]$ with $B^{k}=B^{k}_{1}\nabla^{\omega}_{X^{k}}+B^{k}_{0}$ is a first order differential operators with $\mathcal{C}^{\infty}(\omega;\mathcal{L}(\mathcal{F}))$-coefficients:
\begin{eqnarray*}
  [B^{1},B^{2}]&=&(B^{1}_{1}(\nabla^{\omega}_{X^{1}}B^{2}_{1}))\nabla^{\omega}_{X^{2}}
                -(B^{2}_{1}(\nabla^{\omega}_{X^{2}}B^{1}_{1}))\nabla^{\omega}_{X^{1}}
                   +(B_{1}^{1}B^{2}_{1})\nabla^{\omega}_{[X^{1},X^{2}]}
  \\
               &&+[B^{1}_{0},B^{2}_{0}]+(B^{1}_{1}\nabla^{\omega}_{X^{1}}B^{2}_{0})-(B^{2}_{1}\nabla^{\omega}_{X^{2}}B^{1}_{0})
     +B^{1}_{1}B^{2}_{1}R^{\nabla^{\omega}}(X^{1},X^{2})\,,
\end{eqnarray*}
and this is not true in general when $[B^{1}_{1},B^{2}_{1}]\neq 0$\,.\\
When $\nabla^{\omega}$ is a unitary connection the formal adjoint of $\red B=B_{1}\nabla^{\omega}_{X}+B_{0}$ equals
$$
B^{*}=-\nabla_{X}^{\omega}B_{1}^{*}-(\mathrm{div}\,X) B_{1}^{*}+B_{0}^{*}=-B_{1}^{*}\nabla^{\omega}_{X}+B_{0}^{*}-(\nabla^{\omega}_{X}B_{1}^{*}+(\mathrm{div}\,X) B_{1}^{*})\,.
$$
First order differential operators and their adjoint are fully understood by their local description with a unitary connection $\nabla^{\omega}$ and we can start with \eqref{eq:connectriv}. This will be applied with $\omega=\omega_{m}$\,, $1\leq m\leq M_{I}$\,. In order to use the localization on $\Sigma$ in the variables $(k,\lambda)$\,, it is convenient to associate a new unitary connection  with a finite orthogonal decomposition of the fiber bundle $p_{\mathcal{F}}:\mathcal{F}\to \omega$\,. Namely assume that $(\pi_{n}(k))_{0\leq n\leq N}$ are orthogonal projections in the fiber $\mathcal{F}_{k}$\,, with $\sum_{n=0}^{N}\pi_{n}(k)=\mathrm{Id}_{\mathcal{F}_{k}}$\,, which depend analytically on $k$ (actually $\mathcal{C}^{\infty}$-regularity is only used here). The latter means that $\tilde{\pi}_{n}(k)=U(k)^{-1}\pi_{n}(k)U(k)$ is an orthogonal projection in $\mathcal{H}'$ which is analytic with respect to $k\in \omega$\,, while we know
$$
\pi_{n}(k)^{2}=\pi_{n}(k)
=\pi_{n}(k)^{*}\quad,\quad \sum_{n=0}^{N}\pi_{n}(k)
=\mathrm{Id}_{\mathcal{F}_{k}}\,.
$$
When $\nabla^{\omega}$ is a unitary connection on $p_{\mathcal{F}}:\mathcal{F}\to \omega$\,, endowed with the metric $\langle \,,\,\rangle_{\mathcal{H}'}e^{2V(k)}$\,, this induces a natural connection $\nabla^{\omega,n}$ on $p_{\pi_{n}\mathcal{F}}:\pi_{n}\mathcal{F}\to \omega$ simply given by
$$
\nabla^{\omega,n}=\pi_{n}\circ \nabla^{\omega} \circ \pi_{n}=
\left(\nabla^{\omega} -(\nabla^{\omega} \pi_{n})\right)\big|_{\mathcal{C}^{\infty}(\omega;\pi_{n}\mathcal{F})}=\left(\nabla -\left[(\nabla^{\omega} \pi_{n}),\pi_{n}\right]\right)\big|_{\mathcal{C}^{\infty}(\omega;\pi_{n}\mathcal{F})}\,,
$$
and sometimes called the \emph{adiabatic connection.} The last equality is actually due to
$$
\nabla^{\omega}_{X}(\pi_{n})=\nabla^{\omega}_{X}(\pi_{n}^{2})=(\nabla^{\omega}_{X}\pi_{n})\pi_{n}+\pi_{n}(\nabla^{\omega}_{X}\pi_{n})
$$
which implies 
$$
\pi_{n}(\nabla^{\omega}_{X}\pi_{n})\pi_{n}=0\,.
$$
Then
$$
\nabla^{\omega,\pi}=\oplus_{n=0}^{N}\nabla^{\omega,n}=\oplus_{n=0}^{N}\pi_{n}\nabla^{\omega} \pi_{n}
=\nabla^{\omega}-\sum_{n=1}^{N}(\nabla^{\omega} \pi_{n})\pi_{n}
$$
is a new unitary connexion, $\nabla^{\omega,\pi}$\,, such that for any block diagonal $A\in \mathcal{C}^{\infty}(\omega;\mathcal{L}(\mathcal{F}))$\,, $A=\sum_{n=1}^{N}\pi_{n}A_{n}\pi_{n}$\,,
\begin{equation}
  \label{eq:blockconnex}
\nabla^{\omega,\pi}A=\sum_{n=0}^{N}\pi_{n}\nabla^{\omega} A_{n}\pi_{n}=\sum_{n=0}^{N}\pi_{n}\nabla^{\omega,n}A_{n}\pi_{n}\,,
\end{equation}
and $\nabla_{X}^{\omega,\pi}A$ remains block diagonal for any $X\in \mathcal{C}^{\infty}(\omega;TM)$\,.\\
Actually any finite collection of unitary connections $(\nabla^{\omega,n})_{0\leq n\leq N}$ on the fiber bundles $p_{\pi_{n}\mathcal{F}}:\pi_{n}\mathcal{F}\to \omega$\,, endowed with the metric $\langle~\,,\,~\rangle_{\mathcal{H}'}e^{2V(k)}$\,, gives rise to a unitary connection $\nabla^{\omega,\pi}=\sum_{n=1}^{N}\pi_{n}\nabla^{\omega,n}\pi_{n}$ on $p_{\mathcal{F}}:\mathcal{F}\to \omega$ with the same property \eqref{eq:blockconnex}. Such a unitary connection $\nabla^{\omega,\pi}$ or alternatively $\nabla^{\omega}=\nabla^{\omega,\pi}$ is characterized by
$$
\forall n\in \left\{1,\ldots,N\right\}\,,\quad \nabla^{\omega}\pi_{n}=0\quad\text{in}\quad \mathcal{C}^{\infty}(\omega;\mathcal{L}(\mathcal{F}))
$$
if we notice $\pi_{0}=\mathrm{Id}_{\mathcal{F}}-\sum_{n=1}^{N}\pi_{n}$\,.

A first version of the operator $A_{I}$\,, essentially as presented in \cite{GeNi} is done as follows.
We use the open covering $\bigcup_{(m,n)\in L(I)}\Omega_{m,n}$\,, $\red\Omega_{m,n}=\omega_{m}\times J_{m,n}$\,, of $p_{\Sigma\to\rz}^{-1}(\overline{I})$ given by
Proposition~\ref{pr:covering}.
Above the open set $\omega_{m}$\,, $1\leq m\leq M_{I}$\,, use the decomposition $\mathrm{Id}_{\mathcal{F}_{k}}=\oplus_{n=1}^{N_{m}}1_{J_{m,n}}(H_{0}(k))\oplus 1_{J_{m,0}}(H_{0}(k))=\oplus_{n=0}^{N_{m}}\pi_{m,n}(k)$\,, with $J_{m,0}=\rz\setminus \bigcup_{n=1}^{N_{m}}J_{m,n}$ and
\begin{equation}
  \label{eq:defpimn}
  \pi_{m,n}(k)=1_{J_{m,n}}(H_{0}(k))\,,\quad 0\leq n\leq N_{m}\,,~1\leq m\leq M_{I}\,.
\end{equation}
Consider the unitary connection $\nabla^{\omega_{m},\pi}=\sum_{n=0}^{N_{m}}\pi_{m,n}\nabla^{m,n} \pi_{m,n}$\,, with $\nabla^{m,n}$ a unitary connection defined on $p_{\pi_{m,n}\mathcal{F}}:\pi_{m,n}\mathcal{F}\to \omega_{m}$ by 
\begin{itemize}
\item for $n=0$\,, $\nabla^{m,0}=\pi_{m,0}\nabla^{\omega_{m}}\pi_{m,0}$ with $\nabla^{\omega_{m}}$ defined by \eqref{eq:connectriv} with $\omega=\omega_{m}$\,,
\item for $n\in \left\{1,\ldots, N_{m}\right\}$\,, $\nabla^{m,n}_{X}=W_{m,n}XW_{m,n}^{*}+dV(X)$ where $W_{m,n}(k)$ is the unitary map from $\cz^{\mu_{m,n}}\sim \pi_{m,n}\mathcal{F}_{k_{0}}$ to $\pi_{m,n}\mathcal{F}_{k}$  introduced in Proposition~\ref{pr:covering}-\textit{(5)}.
\end{itemize}
We consider the operator
\begin{eqnarray}
\label{eq:defAm}
  A_{m}&=&i\sum_{n=0}^{N_{m}}g_{m}(k)\left[\nabla^{m,n}_{X_{m,n}}+\frac{1}{2}(\mathrm{div}\,X_{m,n})\pi_{m,n}\right]g_{m}(k)\\
  &=&i\sum_{n=1}^{N_{m}}g_{m}(k)\chi_{m,n}(H_{0}(k))\left[\nabla^{\omega_{m},\pi}_{X_{m,n}}+\frac{1}{2}(\mathrm{div}\,X_{m,n})\right]\chi_{m,n}(H_{0}(k))g_{m}(k)
\end{eqnarray}
where $X_{m,0}=0$ and $X_{m,n}$ for $1\leq n\leq N_{m}$ is the real analytic
vector field introduced in Proposition~\ref{pr:covering}-\textit{6)}
such that
\begin{equation}
  \label{eq:pptyXkmn}
\forall k\in \omega_{m}\,,\quad \nabla^{m,n}_{X_{m,n}}\left(H_{0}(k)1_{J_{m,n}}(k)\right)\geq \frac{1}{2}1_{J_{m,n}}(H_{0}(k))\,.
\end{equation}
The right-hand side of \eqref{eq:defAm} shows that $A_{m}$ is a first order differential operator with $\mathcal{C}^{\infty}_{comp}(\omega_{m};\mathcal{L}(\mathcal{F}))$ coefficients according to the general setting of Definition~\ref{de:firstorder}. Because $\nabla^{\omega_{m},\pi}$ is a unitary connection, $A_{m}$ is symmetric on $\mathcal{C}^{\infty}_{comp}(M;\mathcal{F})$ while the commmutator $[H_{0},iA_{m}]$ satisfies
$$
\left[H_{0},iA_{m}\right]=g_{m}(k)\left[\sum_{n=1}^{N_{m}}
\nabla^{m,n}_{X_{m,n}}\left(H_{0}(k)1_{J_{m,n}}(H_{0}(k)\right)\right]g_{m}(k)\geq \frac{1}{2}g_{m}^{2}(k)1_{\overline{I}}(H_{0}(k))\,.
$$
Clearly the operator $[H_{0},iA_{m}]=\sum_{n=1}^{N_{m}}\pi_{m,n}B_{m,n}\pi_{m,n}$ with $B_{m,n}\in \mathcal{C}^{\infty}_{comp}(\omega_{m}; \mathcal{L}(\pi_{m,n}\mathcal{F}))$\,, $\supp B_{m,n}\subset \supp g_{m}$\,, and iterated commutators $\mathrm{ad}_{iA_{m}}^{k}$ preserve this class.\\
By summing $A_{I}=\sum_{m=1}^{M_{I}}A_{m}$ we obtain
$$
\left[H_{0},iA_{I}\right]\geq \frac{1}{2}\sum_{m=1}^{M_{I}}g_{m}(k)^{2}1_{\overline{I}}(H_{0}(k))=\frac{1}{2} 1_{\overline{I}}(H_{0})\,.
$$
So everything seems to be done here, and this is where we definitely omitted an argument in \cite{GeNi}\,. Actually when $\omega_{m_{1}}\cap \omega_{m_{2}}\neq \emptyset$ or more precisely when $g_{m_{1}}g_{m_{2}}\not\equiv 0$\,, the block diagonal decomposition $\oplus_{n=0}^{N_{m}}1_{J_{m,n}}(H_{0}(k))$  do not coincide for $m=m_{1}$ and $m=m_{2}$\,. While considering the iterated commutator $[A_{m_{j}},[A_{m_{1}},H_{0}]]$\,, $j=1,2$\,, we must check that $[A_{m_{1}},H_{0}]$
has the local structure $\sum_{n=1}^{N_{m}}\pi_{m,n}B_{m,n}\pi_{m,n}$\,, $B_{m,n}\in \mathcal{C}^{\infty}_{comp}(\omega_{m};\mathcal{L}(\pi_{m,n}\mathcal{F}))$\,, both for $m=m_{1}$ and $m=m_{2}$\,. And this is not true in general. This can be corrected by modifying the unitary connections $\nabla^{\omega_{m},\pi}$. But the additional terms deteriorate Mourre's inequality which can be finally recovered only on {\red a} small energy interval $[\lambda-\varepsilon,\lambda+\varepsilon]$\,, $\lambda\in \overline{I}$ and $\varepsilon>0$ small enough. This is done in the following sections.

\section{Modified unitary connection}
\label{sec:modconn}
Here we change the unitary connection $\nabla^{\omega_{m},\pi}=\sum_{n=0}^{N_{m}}\pi_{m,n}\nabla^{\omega_{m}}\pi_{m,n}$ into a new unitary connection $\tilde{\nabla}^{\omega_{m},\pi}=\sum_{n=0}^{N_{m}}\pi_{m,n}\tilde{\nabla}^{m,n}\pi_{m,n}$ {\red whose} the construction depends on the partition of unity \eqref{eq:partunitM}\eqref{eq:partunitSig}\eqref{eq:chi1J}. Remember the definition~\eqref{eq:defpimn} of $\pi_{m,n}$\,. Actually it suffices to specify the new connection $\tilde{\nabla}^{m,n}$ on the fiber bundle $p_{\pi_{m,n}\mathcal{F}}:\pi_{m,n}\mathcal{F}\to \omega_{m}$\,, for every fixed pair $(m,n)$\,, but all the possible non empty intersections $\omega_{\alpha}=\bigcap_{m'\in \alpha}\omega_{m'}$\,, for $m\in \alpha\subset \left\{1,\ldots, M_{I}\right\}$\,, $\emptyset\neq \omega_{\alpha}\subset \omega_{m}$\,, must be considered.  The new operator $\tilde{A}_{I}$ defined like $A_{I}=\sum_{m=1}^{M_{I}}A_{m}$ with this new unitary connection will satisfy for all pairs $(m',n')\in L(I)$\,,
\begin{equation}
  \label{eq:commBmn}
\mathrm{ad_{\tilde{A}_{I}}}B_{m',n'}\in \mathcal{C}^{\infty}_{comp}(\omega_{m'};\mathcal{L}(\pi_{m',n'}\mathcal{F}))\quad\text{with}\quad \supp \mathrm{ad_{\tilde{A}_{I}}}B_{m',n'}\subset \supp g_{m'}
\end{equation}
for any $B_{m',n'}\in \mathcal{C}^{\infty}_{comp}(\omega_{m'};\mathcal{L}(\pi_{m',n'}\mathcal{F}))$\,,  with $\supp B_{m',n'}\subset \supp g_{m'}\Subset \omega_{m'}$\,, and we will deduce
$$
\forall k\in \nz\,,\quad \mathrm{ad}_{i\tilde{A}_{I}}^{k}H_{0}\in \mathcal{L}(L^{2}(M;\mathcal{F}))\,.
$$
For the moment we still work with first order differential operators with $\mathcal{C}^{\infty}(M;\mathcal{L}(\mathcal{F}))$ coefficients (see Definition~\ref{de:firstorder}),  acting on $\mathcal{C}^{\infty}_{comp}(M;\mathcal{F})$\,. The question of their self-adjoint realizations as well as the functional analytic meaning of commutators will be discussed later.\\
By expanding \eqref{eq:commBmn} with
$$
\tilde{A}_{I}=i\sum_{(m,n)\in L(I)}g_{m}(k)\bigg(\tilde{\nabla}_{X_{m,n}}^{m,n}+\frac{(\mathrm{div}\,X_{m,n})}{2}\pi_{m,n}\bigg)g_{m}(k) \quad (\text{remember}~X_{m,0}=0)
$$
the commutator equals
\begin{eqnarray*}
&&[i\tilde{A}_{I},B_{m',n'}]=-\sum_{(m,n)\in L(I)}g_{m}(k)\left[\bigg(\tilde{\nabla}_{X_{m,n}}^{m,n}+\frac{(\mathrm{div}\,X_{m,n})}{2}\pi_{m,n}\bigg),B_{m',n'}\right]g_{m}(k)\\
&&\hspace{1cm}= -\sum_{\tiny
  \begin{array}[t]{c}
    \Omega_{m,n}\cap \Omega_{m',n'}\cap \Sigma\neq \emptyset\\
  (m,n)\in L(I)\end{array}
  }g_{m}(k)\left[\bigg(\tilde{\nabla}_{X_{m,n}}^{m,n}+\frac{(\mathrm{div}\,X_{m,n})}{2}\pi_{m,n}\bigg),\pi_{m',n'}B_{m',n'}\pi_{m',n'}\right]g_{m}(k)\,.
\end{eqnarray*}
for any $(m',n')\in L(I)$\,, where we recall $\Omega_{m,n}=\omega_{m}\times J_{m,n}$ for $(m,n)\in L(I)$ as in Proposition~\ref{pr:covering} with $\supp g_{m}\subset \omega_{m}$ and $1_{J_{m,n}}(H_{0}(k))=\pi_{m,n}(k)$ for $k\in \omega_{m}$\,. 
In the above sum, expand  the commutator $[\tilde{\nabla}^{m,n}_{X_{m,n}},\pi_{m',n'}B_{m',n'}\pi_{m',n'}]$ {\red as:}
\begin{equation*}
  [\tilde{\nabla}^{m,n}_{X_{m,n}},\pi_{m',n}]B_{m',n'}\pi_{m',n'}+
  \pi_{m',n'}[\tilde{\nabla}^{m,n}_{X_{m,n}},B_{m',n}]\pi_{m',n'}+
  \pi_{m',n'}B_{m',n}[\tilde{\nabla}^{m,n}_{X_{m,n}}\,,\pi_{m',n'}]
\end{equation*}
while, as operator valued differential operators on $\omega_{m}\cap \omega_{m'}$\,,
\begin{eqnarray*}
[\tilde{\nabla}^{m,n}_{X_{m,n}},\pi_{m',n'}]&=&\pi_{m,n}\tilde{\nabla}_{X_{m,n}}^{m,n}\pi_{m,n}\pi_{m',n'}-\pi_{m',n'}\pi_{m,n}\tilde{\nabla}_{X_{m,n}}^{m,n}\pi_{m,n}
  \\
&=&  \left[\tilde{\nabla}^{m,n}_{X_{m,n}}, \pi_{m,n}\pi_{m',n'}\right]
=\tilde{\nabla}^{m,n}_{X_{m,n}}(\pi_{m,n}\pi_{m',n'})\,.
\end{eqnarray*}
Therefore the condition \eqref{eq:commBmn} is satisfied when
\begin{equation}
  \label{eq:annulconn}
  \tilde{\nabla}^{m,n}(\pi_{m,n}\pi_{m',n'})=0 \quad\text{above}~\supp g_{m}\cap \supp g_{m'}\,,
\end{equation}
for all pairs $(m,n)$ and $(m',n')$ in $L(I)$ such that $\Omega_{m,n}\cap \Omega_{m',n'}\cap \Sigma\neq \emptyset$\,.
This can be discussed by using the order introduced in Proposition~\ref{pr:order}  after the classification of cases in Proposition~\ref{pr:covering}-\textit{(7)} and Definition~\ref{de:asym}.\\
\begin{proposition}
  \label{pr:AI1I}
  Let $I$ and $\tilde{I}$ be fixed intervals such that $I\Subset \tilde{I}\subset \rz\setminus\tau$ and let the partition of unity given by \eqref{eq:partunitM}\,\eqref{eq:partunitSig}\,\eqref{eq:chi1J} after Proposition~\ref{pr:covering} be fixed.\\
  The unitary connections $(\tilde{\nabla}^{m,n})_{(m,n)\in L(I)}$ on the vector bundles $p_{\pi_{m,n}\mathcal{F}}:\pi_{m,n}\mathcal{F}\to \omega_{m}$ can be defined such that \eqref{eq:annulconn} holds true for all pairs $(m,n)$ and $(m',n')$ in $L(I)$\,.
  These connections satisy the additional property that when $(k_{0},\lambda_{0})\in \Omega_{m,n}\cap p_{\Sigma\to\rz}(\overline{I})\subset \Sigma$ and the set $\left\{(m',n')\in L(I)\,, (k_{0},\lambda_{0})\in \Omega_{m',n'}\right\}=\left\{(m_{\ell},n_{\ell})\,, 1\leq \ell\leq L\right\}$ is ordered like in Proposition~\ref{pr:order}\,, there exists an open neighborhood $\mathcal{V}_{k_{0}}$ of $k_{0}$ such that
$$
\pi_{m_{L},n_{L}}\circ\tilde{\nabla}^{m,n}\circ\pi_{m_{L},n_{L}}=\pi_{m_{L},n_{L}}\circ \nabla^{m,n}\circ \pi_{m_{L},n_{L}}\quad\text{on}~\mathcal{V}_{k_{0}}\,.
$$
\end{proposition}
The proof will be done in several steps. It can  be done for any fixed pair $(m,n)\in L(I)$\,.  We firstly construct  a unitary connection $\nabla^{\alpha}$ on the restricted vector bundle $p_{\pi_{m,n}\mathcal{F}}: \pi_{m,n}\mathcal{F}\to \omega_{\alpha}$ where $m\in\alpha\subset \left\{1,\ldots,M_{I_{1}}\right\}$ and $\omega_{\alpha}=\ccap_{m'\in \alpha}\omega_{m'}\neq \emptyset$\,. Then we use a partition of unity in order to glue all these connections $\nabla^{\alpha}$\,.
\begin{lemma}
\label{le:Cstar}
  Let $\mathcal{P}_{\mathcal{I}}=\left\{\pi_{i}, i\in \mathcal{I}\right\}$ be a finite non empty set of self-adjoint projectors $\pi_{i}^{2}=\pi_{i}=\pi_{i}^{*}$ in a unital \underline{commutative} $\mathcal{C}^{*}$-algebra $\mathcal{A}$\,.
  For $\gamma\subset \mathcal{I}$ set $\pi_{\gamma}=\prod_{i\in \gamma}\pi_{i}$ with the convention $\pi_{\emptyset}=1_{\mathcal{A}}$\,.  Assume that $\pi\in \mathcal{P}_{\mathcal{I}}$ implies $1-\pi\in \mathcal{P}_{\mathcal{I}}$\,. Then the $\mathcal{C}^{*}$-algebra generated by $\mathcal{P}_{\mathcal{I}}$\,, $\mathcal{C}^{*}(\mathcal{P}_{\mathcal{I}})$\,, is finite dimensional and a basis is given by the family, $\Gamma$\,, made of $\pi_{\gamma}\neq 0$ such that for all $i\in \mathcal{I}$\,, $\pi_{i}\pi_{\gamma}=0$ when $i\not\in \gamma$\,. In particular for any $i\in \mathcal{I}$\,, $\pi_{i}$ admits a unique decomposition $\pi_{i}=\sum_{\gamma\in \Gamma}c_{i,\gamma}\pi_{\gamma}$ with $c_{i,\gamma}\in \left\{0,1\right\}$\,.
\end{lemma}
\begin{proof}
  The commutative $\mathcal{C}^{*}$-algebra generated by $\mathcal{P}_{\mathcal{I}}$ is the set of linear combinations
 $$
\sum_{\gamma\subset \mathcal{I}}c_{\gamma}\pi_{\gamma}\,,\quad c_{\gamma}\in \cz\,,
$$
where all the $\pi_{\gamma}$'s are self-adjoint projectors.\\
For any $\gamma\subset \mathcal{I}$ and any $i\in \mathcal{I}$ we can write
$$
\pi_{\gamma}=\pi_{i}\pi_{\gamma}+(1_{\mathcal{A}}-\pi_{i})\pi_{\gamma}=\pi_{i}\pi_{\gamma}+\pi_{i'}\pi_{\gamma}
$$
with $\pi_{i'}=1_{\mathcal{A}}-\pi_{i}\in \mathcal{P}_{\mathcal{I}}$\,. We deduce that any $\pi_{\gamma}$\,, $\gamma\subset \mathcal{I}$\,, can be written as a finite sum of $\pi_{\gamma'}$\,, $\gamma'\subset \mathcal{I}$\,, such that $\pi_{i}\pi_{\gamma'}=0$ for all $i\in \mathcal{I}\setminus \gamma'$\,.\\
Our family $\Gamma$\,, is thus a linear generating family of $\mathcal{C}^{*}(\mathcal{P}_{\mathcal{I}})$\,. It is linearly independent because
$\sum_{\gamma\in \mathcal{F}}c_{\gamma}\pi_{\gamma}=0$ implies  $c_{\gamma_{1}}\pi_{\gamma_{1}}=\pi_{\gamma_{1}}\left[\sum_{\gamma\in \mathcal{F}}c_{\gamma}\pi_{\gamma}\right]=0$ and therefore $c_{\gamma_{1}}=0$ for all $\gamma_{1}\in \Gamma$\,.
\end{proof}
We now construct the connection $\nabla^{\alpha}$ on $p_{\pi_{m,n}\mathcal{F}}: \pi_{m,n}\mathcal{F}\to \omega_{\alpha}$ when $\omega_{\alpha}=\bigcap_{m'\in \alpha}\omega_{m'}\neq \emptyset$ for $m\in \alpha{\red\subset} \left\{1,\ldots, M_{I}\right\}$\,.
\begin{proposition}
  \label{pr:nablaalpha}
  Fix $\alpha\subset\left\{1,\ldots,M_{I}\right\}$ and $(m,n)\in L(I)$\,.
  Assume $m\in \alpha\subset\left\{1,\ldots,M_{I}\right\}$ with $\omega_{\alpha}=\bigcap_{m'\in \alpha}\omega_{m'}\neq \emptyset$\,. There exists a linearly independent finite family $(\pi_{\gamma})_{\gamma\in \Gamma_{\alpha}}$ of spectral projections $\pi_{\gamma}(k)=1_{J_{\gamma}}(H_{0}(k))$\,, $J_{\gamma}\subset J_{m,n}$\,, which are analytic with respect to $k\in \omega_{\alpha}$\,, and such that
  \begin{itemize}
  \item $\pi_{\gamma}\pi_{\gamma'}=0$ if $\gamma\neq \gamma'$\,, $\gamma,\gamma'\in \Gamma_{\alpha}$\,,
  \item for all {\red pairs} $(m',n')$\,, $m'\in \alpha$\,, $0\leq n'\leq N_{m'}$\,, the spectral projection $\pi_{m,n}\pi_{m',n'}$ admits a unique decomposition
$$
\pi_{m,n}\pi_{m',n'}=\sum_{\gamma\in \Gamma_{\alpha}}c_{m',n',\gamma}\pi_{\gamma}\quad,\quad c_{m',n',\gamma}\in \left\{0,1\right\}\,.
$$
\end{itemize}
The unitary connection $\nabla^{\alpha}=\sum_{\gamma\in \Gamma_{\alpha}}\pi_{\gamma}\circ \nabla^{m,n}\circ \pi_{\gamma}$ on $p_{\pi_{m,n}\mathcal{F}}:\pi_{m,n}\mathcal{F}\to \omega_{\alpha}$ then satisfies
$$
\forall m'\in \alpha\,, \forall n'\in \left\{0,\ldots, N_{k_{m'}}\right\}\,,\quad \nabla^{\alpha}(\pi_{m,n}\pi_{m',n'})=0\quad\text{on}~\omega_{\alpha}\,.
$$
If $J_{\gamma}\cap \overline{I}\neq \emptyset$ for $\gamma\in \Gamma_{\alpha}$\,, then $\left\{(m',n')\,, m'\in \alpha, 0\leq n'\leq N_{m'}, J_{\gamma}\subset J_{m',n'} \right\}\subset L(I)$ can be ordered like in Proposition~\ref{pr:order} and written $\left\{(m_{\ell},n_{\ell})\,, 1\leq \ell\leq L_{\gamma}\right\}$ with
$$
J_{\gamma}=\ccap_{
    \ell=1}^{L_{\gamma}}(J_{m,n}\cap J_{m_{\ell},n_{\ell}})=J_{m,n}\cap J_{m_{L_{\gamma}},n_{L_{\gamma}}} \quad,\quad \pi_{\gamma}=\pi_{m,n}\pi_{m_{L_{\gamma}},n_{_{L_{\gamma}}}}\,.
$$
\end{proposition}
\begin{proof}
  Let us work in the unital commutative $\mathcal{C}^{*}$-algebra, $\mathcal{A}$\,, of continuous functions of $f(k,H_{0}(k))$\,, $k\in \omega_{\alpha}$\,, such that $\pi_{m,n}(k)f(k,H_{0}(k))=f(k,H_{0}(k))$ with the unit $\pi_{m,n}$\,.\\
 For any $m'\in \alpha$\,, we know the fiberwise identity $\oplus_{n'=0}^{N_{k_{m'}}}\pi_{m',n'}(k)=\mathrm{Id}_{\mathcal{F}_{k}}$ for $k\in \omega_{\alpha}\subset \omega_{m'}$ and
  $$
  \sum_{n'=0}^{N_{k_{m'}}}\pi_{m,n}\pi_{m',n'}=\pi_{m,n}=1_{\mathcal{A}}\,.
  $$
  By {\red setting} $\mathcal{I}_{0}=\left\{(m',n')\,, m'\in \alpha, 0\leq n'\leq N_{k_{m'}}\right\}$ and $\mathcal{P}_{\mathcal{I}_{0}}=\left\{\pi_{m,n}\pi_{m',n'}\,, (m',n')\in \mathcal{I}_{0}\right\}$ and $\mathcal{P}_{\mathcal{I}}=\mathcal{P}_{\mathcal{I}_{0}}\bigcup \left\{1_{\mathcal{A}}-\pi_{m,n}\pi_{m',n'}\,, (m',n')\in \mathcal{I}_{0}\right\}$\,, $\mathcal{I}=\mathcal{I}_{0}\times \left\{0,1\right\}$\,, we deduce
$$
\mathcal{C}^{*}(\mathcal{P}_{\mathcal{I}_{0}})=\mathcal{C}^{*}(\mathcal{P}_{\mathcal{I}})\,.
$$
The family $\mathcal{P}_{\mathcal{I}}$ is stable by $\pi\to 1_{\mathcal{A}}-\pi$ and Lemma~\ref{le:Cstar} can be applied.\\
By construction the basis $(\pi_{\gamma})_{\gamma\in \Gamma_{\alpha}}$ of $\mathcal{C}^{*}(\mathcal{P}_{\mathcal{I}_{0}})=\mathcal{C}^{*}(\mathcal{P}_{\mathcal{I}})$\,, is made of products
\begin{equation}
  \label{eq:formpigamma}
\prod_{(m',n')\in \gamma_{0}}(\pi_{m,n}\pi_{m',n'})\circ \prod_{(m',n')\in \gamma_{1}}(\pi_{m,n}-\pi_{m,n}\pi_{m',n'})
\end{equation}
for $\gamma_{0},\gamma_{1}\subset \mathcal{I}_{0}$\,. They can be written $1_{J_{\gamma}}(H_{0}(k))$ with $J_{\gamma}\subset J_{m,n}$ and  they are all analytic with respect to $k\in \omega_{\alpha}$\,.\\
The property of the unitary connection $\nabla^{\alpha}$ is due to $\pi_{\gamma}(\pi_{m,n}\pi_{m',n'})\in \left\{\pi_{\gamma},0\right\}$ for all $\gamma\in \Gamma_{\alpha}$\,, $m'\in \alpha$ and all $n'\in \left\{0,\ldots, N_{k_{m'}}\right\}$\,.\\
When $J_{\gamma}\cap \overline{I}\neq \emptyset$ with $\pi_{\gamma}=1_{J_{\gamma}}(H_{0}(k))$\,,  the pairs $(m',n')$ involved in the first factor of
\eqref{eq:formpigamma} must satisfy $n'\geq 1$ while in the second factor the pairs must be $(m',0)$\,. With $\pi_{m,n}-\pi_{m,n}\pi_{m',0}=\pi_{m,n}\sum_{n'=1}^{N_{m'}}\pi_{m',n'}$\,,  this implies that $\pi_{\gamma}$ is a positive linear combination of products of $\pi_{m',n'}$\,, $m'\in \alpha$\,, $(m',n')\in L(I)$\,. But the two properties of the basis $(\pi_{\gamma})_{\gamma\in \Gamma_{\alpha}}$ says that it must be the product of the all the $\pi_{m_{\ell},n_{\ell}}$\,, $1\leq \ell\leq L_{\gamma}$ with $m_{\ell}\in \alpha$\,, $(m_{\ell},n_{\ell})\in L(I)$ and $J_{\gamma}\subset J_{m_{\ell},n_{\ell}}$\,. For such pairs $(m',n')$ we have
$$
\forall k\in \omega_{\alpha}\,,\quad 0\neq \pi_{\gamma}(k)\leq \pi_{m',n'}(k)\,,
$$
and the intersection $\bigcap_{\ell=1}^{L_{\gamma}}\Omega_{m_{\ell}}\cap \Sigma\supset (\omega_{\alpha}\times J_{\gamma})\cap \Sigma\neq \emptyset$\,. By ordering the set $\left\{(m_{\ell},n_{\ell})\,, 1\leq \ell\leq L_{\gamma}\right\}$ like in Proposition~\ref{pr:order} we get $J_{\gamma}=J_{m_{L_{\lambda}},n_{L_{\gamma}}}$ and $\pi_{\gamma}(k)=\pi_{m_{L_{\gamma}},n_{L_{\gamma}}}(k)$ for $k\in \omega_{\alpha}$\,.
\end{proof}
\begin{proof}[Proof of Proposition~\ref{pr:AI1I}]
Fix $m\in \left\{1,\ldots,M_{I}\right\}$\,.  The set $K=\bigcup_{m'\neq m}\supp (g_{m}\cap \supp g_{m'})$ is a  compact subset of $\omega_{m}$\,. For any $\alpha\subset \left\{1,\ldots, M_{I}\right\}$ such that $m\in \alpha$\,, $\sharp \alpha\geq 2$\,, set
 $$
 \tilde{\omega}_{\alpha}=\omega_{\alpha}\setminus \ccup_{m'\not\in \alpha}\supp g_{m'}
 =\ccap_{m'\in \alpha}\omega_{m'}\setminus (\ccup_{m'\not \in \alpha}\supp g_{m'})\,.
 $$
 These are open sets and they make an open covering of $K$\,,
$$
K\subset \ccup_{\tiny
  \begin{array}[c]{c}
    \alpha\subset \left\{1,\ldots, M_{I}\right\}\\
    m\in \alpha\\
    \sharp \alpha\geq 2
  \end{array}
}\tilde{\omega}_{\alpha}\,.
$$
Therefore there exists a partition of unity
$$
\sum_{m\in \alpha\subset \left\{1,\ldots, M_{I}\right\}}\Theta_{\alpha}^{2}\equiv 1\quad \text{on}~\omega_{m}
$$
with
\begin{itemize}
\item $\Theta_{\alpha}\in \mathcal{C}^{\infty}_{comp}(\tilde{\omega}_{\alpha})$ when $\sharp \alpha\geq 2$\,,
\item $\Theta_{\left\{m\right\}}\in \mathcal{C}^{\infty}(\omega_{m})$ and $\supp \Theta_{\left\{m\right\}}\subset \omega_{m}\setminus K$\,.
\end{itemize}
The unitary connection $\tilde{\nabla}$ on $p_{\pi_{m,n}\mathcal{F}}:\pi_{m,n}\mathcal{F}\to \omega_{m}$ is defined as
$$
\tilde{\nabla}^{m,n}=\sum_{m\in \alpha \subset\left\{1,\ldots,M_{I}\right\}}
\Theta_{\alpha}\nabla^{\alpha}\Theta_{\alpha}
$$
where $\nabla^{\alpha}$ is the connection above $\omega_{\alpha}\supset\tilde{\omega}_{\alpha}$ introduced in Proposition~\ref{pr:nablaalpha}.
For any $m'\in \left\{1,\ldots, M_{I}\right\}$ such that $\supp g_{m}\cap \supp g_{m'}\neq \emptyset$ we get
\begin{eqnarray*}
  &&
\tilde{\nabla}^{m,n}=\sum_{\left\{m,m'\right\}\subset \alpha \subset \left\{1,\ldots M_{I}\right\}}\Theta_{\alpha}\nabla^{\alpha}\Theta_{\alpha}
  \\
\text{with}&& \sum_{\left\{m,m'\right\}\subset \alpha \subset \left\{1,\ldots M_{I}\right\}}\Theta_{\alpha}^{2}\equiv 1
\end{eqnarray*}
on a neighborhood of $\mathcal{V}_{m,m'}$ of $\supp g_{m}\cap \supp g_{m'}$\,.  This implies for any $n'\in \left\{1,\ldots,N_{m'}\right\}$\,,
$$
\tilde{\nabla}^{m,n}(\pi_{m,n}\pi_{m',n'})=\sum_{\left\{m,m'\right\}\subset \alpha \subset \left\{1,\ldots M_{I}\right\}}\Theta_{\alpha}^{2}\nabla^{\alpha}(\pi_{m,n}\pi_{m',n'}\big|_{\omega_{\alpha}})\quad\text{on}~\mathcal{V}_{m,m'}\,,
$$
where Proposition~\ref{pr:nablaalpha} says $\nabla^{\alpha}(\pi_{m}\pi_{m'}\big|_{\omega_{\alpha}})=0$ when $m'\in \alpha$\,.\\
Therefore  \eqref{eq:annulconn} holds true for all pairs $(m,n)$ and $(m',n')$ in $L(I)$\,.\\
For the second property we again fix $m\in \left\{1,\ldots, M_{I}\right\}$\,, 
$(k_{0},\lambda_{0})\in \Omega_{m,n}\cap p_{\Sigma\to\rz}^{-1}(\overline{I})\subset \Sigma$\,, and consider the set
$$
\left\{(m',n')\in L(I)\,,\, (k_{0},\lambda_{0})\in \Omega_{m',n'}\right\}=\left\{(m_{\ell},n_{\ell})\,,\, 1\leq \ell\leq L\right\}\,.
$$
The intersection $\bigcap_{\ell=1}^{L}\Omega_{m_{\ell},n_{\ell}}\cap \Sigma\ni (k_{0},\lambda_{0})$ and the family $((m_{\ell},n_{\ell}))_{1\leq \ell\leq L}$ can be ordered according to Proposition~\ref{pr:order}.\\
Define
$$
\alpha_{0}=\left\{m_{\ell}\,,\, 1\leq \ell\leq L\right\}
$$
where obviously $m\in \alpha_{0}$ and $k_{0}\in \omega_{\alpha_{0}}$\,.
 A stronger property says that $k_{0}\in \omega_{\alpha}$ for $m\in \alpha\subset\left\{1,\ldots,M_{I}\right\}$\,, implies $\alpha\subset \alpha_{0}$\,: Actually if $k_{0}\in \omega_{m'}$\,, then the condition $(k_{0},\lambda_{0})\in \Sigma\cap (\omega_{m'}\times \overline{I})$ implies the existence of  $(m',n')\in L(I)$ such that $\lambda_{0}\in J_{m',n'}$\,. Thus $k_{0}\in \omega_{m'}$ implies $m'\in \alpha_{0}$ and therefore $k_{0}\not \in \omega_{\alpha}$ for $\alpha\not\subset \alpha_{0}$\,. Because $\supp \Theta_{\alpha}\Subset \omega_{\alpha}$ for $m\in \alpha\not\subset \alpha_{0}$\,, $\mathcal{V}_{k_{0}}=\omega_{\alpha_{0}}\setminus \ccup_{m\in\alpha\not\subset \alpha_{0}}\supp \Theta_{\alpha}$ is an open neighborhood of $k_{0}$\,.
We now write
\begin{eqnarray*}
  && \tilde{\nabla}^{m,n}=\sum_{m\in \alpha\subset \alpha_{0}}\Theta_{\alpha}\nabla^{\alpha}\Theta_{\alpha}\quad\text{on}~\mathcal{V}_{k_{0}}\,,\\
  \text{and}&& \sum_{m\in \alpha\subset \alpha_{0}}\Theta_{\alpha}^{2}\equiv 1\quad \text{on}~\mathcal{V}_{k_{0}}\,.
\end{eqnarray*}
With the notations of Proposition~\ref{pr:nablaalpha}, for  $m\in\alpha\subset \alpha_{0}$ there exists $\gamma\in \Gamma_{\alpha}$  such that $J_{\gamma}$ contains $\lambda_{0}\in \text{Spec}(H_{0}(k_{0}))\cap \overline{I}$\,. For such a $\gamma$\,, Proposition~\ref{pr:nablaalpha} provides a pair $(m_{\alpha,\gamma},n_{\alpha,\gamma})\in L(I)$\,, $m_{\alpha,\gamma}\in \alpha$ such that $\pi_{\gamma}=\pi_{m,n}\pi_{m_{\alpha,\gamma},n_{\alpha,\gamma}}$\,, $(k_{0},\lambda_{0})\in \Omega_{m_{\alpha,\gamma},n_{\alpha,\gamma}}$\,, and $\pi_{m_{\alpha,\gamma},n_{\alpha,\gamma}}\circ \nabla^{\alpha}\circ \pi_{m_{\alpha,\gamma},n_{\alpha,\gamma}}=\pi_{m_{\alpha,\gamma},n_{\alpha_{\gamma}}}\circ\nabla^{m,n}\circ\pi_{m_{\alpha,\gamma},n_{\alpha,\gamma}}$
on $\omega_{\alpha}$ and therefore on $\mathcal{V}_{k_{0}}\subset \omega_{\alpha_{0}}\subset \omega_{\alpha}$\,. Because $(k_{0},\lambda_{0})\in \Omega_{m_{\alpha,\gamma},n_{\alpha,\gamma}}$\,, there exist $\ell\in \left\{1,\ldots L\right\}$ such that $(m_{\alpha,\gamma},n_{\alpha,\gamma})=(m_{\ell},n_{\ell})$ and the ordering of Proposition~\ref{pr:order} says $\pi_{m_{\alpha,\gamma},n_{\alpha,\gamma}}\pi_{m_{L},n_{L}}=\pi_{m_{\ell},n_{\ell}}\pi_{m_{L},n_{L}}=\pi_{m_{L},n_{L}}$ and
\begin{eqnarray*}
  \pi_{m_{L},n_{L}}\circ \tilde{\nabla}^{m,n}\circ \pi_{m_{L},n_{L}}\big|_{\mathcal{V}_{k_{0}}}
  &=& \sum_{m\in \alpha\subset \alpha_{0}}\Theta_{\alpha}\pi_{m_{L},n_{L}}\pi_{m_{\alpha,\gamma},n_{\alpha,\gamma}}\nabla^{\alpha}\pi_{m_{\alpha,\gamma},n_{\alpha,\gamma}}\pi_{m_{L},n_{L}}\Theta_{\alpha}\big|_{\mathcal{V}_{k_{0}}}\\
  &=& \pi_{m_{L},n_{L}}\circ \left[\sum_{m\in \alpha\subset \alpha_{0}}\Theta_{\alpha}\nabla^{m,n}\Theta_{\alpha}\right]\circ \pi_{m_{L},n_{L}}\big|_{\mathcal{V}_{k_{0}}}\\
  &=& \pi_{m_{L},n_{L}}\circ\nabla^{m,n}\circ\pi_{m_{L},n_{L}}\big|_{\mathcal{V}_{k_{0}}}\,.
\end{eqnarray*}
\end{proof}

\section{The operator $\tilde{A}_{I}$}
\label{sec:opAII}

We specfify here the construction and properties of $\tilde{A}_{I}$ and finish the proof of Theorem~\ref{th:AI1I}.
As mentionned in the beginning of Section~\ref{sec:modconn} the operator $\tilde{A}_{I}$ is given as a first-order differential operator with $\mathcal{C}^{\infty}_{comp}(M;\mathcal{L}(\mathcal{F}))$-coefficients by
$$
\tilde{A}_{I}=i\sum_{m=1}^{M_{I}}\sum_{n=1}^{N_{m}} g_{m}(k)\bigg(\tilde{\nabla}^{m,n}_{X_{m,n}}+\frac{(\mathrm{div}\,X_{m,n})}{2}\pi_{m,n}\bigg)g_{m}(k)
$$
where
\begin{itemize}
\item 
  $\sum_{m=1}^{M_{I}}g_{m}^{2}(k)\equiv 1$  in a neighborhood of $p_{M}(p_{\Sigma\to \rz}^{-1}(\overline{I}))$\,, is the partition of unity subordinate to $p_{M}(\Sigma\cap p_{\Sigma\to\rz}^{-1}(\overline{I}_{1}))\subset \bigcup_{m=1}^{M_{I}}\omega_{m}$\,, introduced in \eqref{eq:partunitM}\eqref{eq:partunitSig}\eqref{eq:chi1J} after Proposition~\ref{pr:covering};
\item $X_{m,n}$ is a vector field such that \eqref{eq:pptyXkmn} holds, of which the local construction was given in Proposition~\ref{pr:covering}-\textit{(6)};
\item $\tilde{\nabla}^{m,n}$ is the modified unitary connection given by Proposition~\ref{pr:AI1I}.
\end{itemize}
By construction $\bigcup_{(m,n)\in L(I)}^{M_{I}}J_{m,n}\Subset \tilde{I}$ and there exists $\chi\in \mathcal{C}^{\infty}_{comp}(\tilde{I};\rz)$ such that
\begin{equation}
  \label{eq:AchiH0}
\tilde{A}_{I}\chi(H_{0})=\tilde{A}_{I}=\chi(H_{0})\tilde{A}_{I} \quad \text{on}~\mathcal{C}^{\infty}_{comp}(M;\mathcal{F})\,.
\end{equation}
As a differential operator we get
$$
[H_{0},iA]=\sum_{(m,n)\in L(I)}g_{m}(k)\pi_{m,n}(k) (\tilde{\nabla}^{m,n}_{X_{m,n}}H_{0}(k))\pi_{m,n}(k)g_{m}(k)=\sum_{(m',n')\in L(I)}B_{m',n'}(k)
$$
with $B_{m',n'}\in \mathcal{C}^{\infty}_{comp}(\omega_{m'};\mathcal{L}(\pi_{m,n}\mathcal{F}))$ with $\supp B_{m',n'}\subset \supp g_{m'}$\,.
But because the connection $\tilde{\nabla}^{m,n}$ satisfies \eqref{eq:annulconn} we deduce that for all $j\in\nz$\,,
$$
\ad_{i\tilde{A}_{I}}^{j}[H_{0},i\tilde{A}_{I}]\in \mathcal{L}(L^{2}(M;\mathcal{F}))\,.
$$
If $\tilde{A}_{I}$ is essentially self-adjoint on $\mathcal{C}^{\infty}_{comp}(M;\mathcal{F})$ and that its self-adjoint extension, still denoted by $\tilde{A}_{I}$\,, satisfies $e^{it\tilde{A}_{I}}D(H_{0})\subset D(H_{0})$\,, for all $t\in\rz$\,, then $H_{0}\in \mathcal{C}^{\infty}(i\tilde{A}_{I})$\,.\\
The second statement is a consequence of \eqref{eq:AchiH0} and it suffices to check the essential self-adjointness of $\tilde{A}_{I}$\,.
Because $p_{M}(\Sigma\cap p_{\Sigma\to \rz}^{-1}\supp \chi)$ is compact with $\chi$ chosen like in \eqref{eq:AchiH0}\,, there exists a smooth domain $\omega$ with $\overline{\omega}$ compact in $M$ such that the coefficients of $\tilde{A}_{I}$ belong to $\mathcal{C}^{\infty}_{comp}(\omega;\mathcal{L}(\mathcal{F}))$\,. Take a finite covering $\overline{\omega}\subset \bigcup_{n=1}^{n_{max}}\omega_{n}$ of coordinate open charts $\omega_{n}$\,. When $\sum_{n=1}^{n_{max}}\theta_{n}^{2}(k)\equiv 1$ in a neighborhood of $\overline{\omega}$ with $\theta_{n}\in \mathcal{C}^{\infty}_{comp}(\omega_{n};\rz)$ consider the operator
\begin{eqnarray*}
  && -\Delta_{\omega}= -\sum_{n=1}^{n_{max}}\theta_{n}\Delta_{\omega_{n}}\theta_{n}\\
  \text{with}&& \Delta_{\omega_{n}}=\sum_{i=1}^{\dim M}(\nabla^{M}_{\frac{\partial}{\partial k_{i}}})^{2}\quad,
\end{eqnarray*}
where $(k_{1},\ldots,k_{\dim M})$ are coordinates on $\omega_{n}$ and $\nabla^{M}$ is the unitary connection defined by \eqref{eq:nablaM}\,. In a neighborhood of $\overline{\omega}$\,, $-\Delta_{\omega}$ is an elliptic second order operator with a scalar principal part. Moreover there exists a constant $C_{\omega}>0$ such that $C_{\omega}-\Delta_{\omega}$ is non negative on $\mathcal{C}^{\infty}_{comp}(\omega;\mathcal{L}(\mathcal{F}))$\,. The standard elliptic regularity on the smooth domain $\overline{\omega}$ (see e.g. \cite{ChPi}) says that the Friedrichs extension of $C_{\omega}-\Delta_{\omega}$\,, which is the Dirichlet realization denoted by $\mathcal{N}=C_{\omega}-\Delta_{\omega}^{D}$, is self-adjoint with domain $H^{2}(\overline{\omega};\mathcal{F})\cap H^{1}_{0}(\omega;\mathcal{F})$ and that $\mathcal{C}^{\infty}_{comp}(\omega;\mathcal{F})$ is a core for $\mathcal{N}=(C_{\omega}-\Delta_{\omega}^{D})$\,.\\
The differential operator $\tilde{A}_{I}$ acts on $\mathcal{C}^{\infty}_{comp}(\omega;\mathcal{F})$ with
\begin{eqnarray*}
   \|\tilde{A}_{I}\varphi\|\leq C\|\mathcal{N}\varphi\|\\
  \left|\langle \tilde{A}_{I}\varphi\,,\, \mathcal{N}\varphi\rangle-\langle \mathcal{N}\varphi\,,\, \tilde{A}_{I}\varphi\rangle\right|\leq C\|\mathcal{N}^{1/2}\varphi\|^{2}
\end{eqnarray*}
for all $\varphi\in \mathcal{C}^{\infty}_{comp}(\omega;\mathcal{F})$\,. The second inequality is due to the fact that $\mathcal{N}$ is a second order differential operator with a scalar principal part so that $[\mathcal{N},\tilde{A}_{I}]$ is a second order differential operator with $\mathcal{C}^{\infty}_{comp}(\omega;\mathcal{L}(\mathcal{F}))$ coefficients. Nelson's commutator theorem (see e.g. \cite{ReSi}) tells us that $\tilde{A}_{I}$ is essentially self-adjoint on $\mathcal{C}^{\infty}_{comp}(\omega;\mathcal{F})$ and therefore on $\mathcal{C}^{\infty}_{comp}(M;\mathcal{F})$\,.\\
The final point is about the local Mourre inequality \eqref{eq:localMourre}. Let $\lambda_{0}\in \overline{I}$ for any $k_{0}\in M$ such that $(k_{0},\lambda_{0})\in \Sigma$\,, Proposition~\ref{pr:AI1I} provides a neighborhood $\mathcal{V}_{0}$ of $k_{0}$ and a pair $(m_{L},n_{L})=(m_{L(k_{0})},n_{L(k_{0})})$ such that $(k_{0},\lambda_{0})\in \Omega_{m_{L},n_{L}}$ 
$$
\forall (m,n)\in L(I)\,,\quad
\pi_{m_{L},n_{L}}\circ \tilde{\nabla}^{m,n}\circ \pi_{m_{L},n_{L}}=\pi_{m_{L},n_{L}}\circ \nabla^{m,n}\pi_{m_{L},n_{L}}\,,
$$
while $\pi_{m_{L},n_{L}}\pi_{m,n}\in \{\pi_{m_{L},n_{L}},0\}$\,.\\
For $k\in \mathcal{V}_{k_{0}}$ we deduce
\begin{eqnarray*}
  \pi_{m_{L},n_{L}}[H_{0},i\tilde{A}_{I}]\pi_{m_{L},n_{L}}&=&\pi_{m_{L},n_{L}}[H_{0},-\sum_{(m,n)\in L(I)}
 g_{m}(m)\pi_{m_{L},n_{L}}\tilde{\nabla}^{m,n}_{X_{m,n}}\pi_{m_{L},n_{L}}g_{m}(k)]\pi_{m_{L},n_{L}}\\
 &=&\pi_{m_{L},n_{L}}[H_{0},-\sum_{(m,n)\in L(I)}g_{m}(k)\nabla^{m,n}g_{m}(k)]\pi_{m_{L},n_{L}}\\
                                                        &\geq &\frac{1}{2}\pi_{m_{L},n_{L}}\left(\sum_{(m,n)\in L(I)}g_{m}(k)^{2}\pi_{m,n}\right)\pi_{m_{L},n_{L}}\\
 &\geq& \frac{1}{2}\pi_{m_{L},n_{L}}=\frac{1}{2}1_{J_{m_{L},n_{L}}}(H_{0}(k))\,.
\end{eqnarray*}
Now $p_{\Sigma\to\rz}^{-1}(\left\{\lambda_{0}\right\})\cap \Sigma$ is compact and $\bigcup_{(k_{0},\lambda_{0})\in \Sigma}\mathcal{V}_{k_{0}}\times J_{m_{L(k_{0})},n_{L(k_{0})}}$ is an open covering of it. We can find $N_{\lambda_{0}}\in \nz$ such that $p_{\Sigma\to\rz}^{-1}(\left\{\lambda_{0}\right\})\cap \Sigma\subset \bigcup_{\nu=1}^{N_{\lambda_{0}}}\mathcal{V}_{k_{\nu}}\times J_{m_{L(k_{\nu})},n_{L(k_{\nu})}}$\,. Therefore we can take $\delta>0$ such that $p_{\Sigma\to\rz}^{-1}([\lambda_{0}-\delta,\lambda_{0}+\delta])\cap \Sigma\subset \bigcup_{\nu=1}^{N_{\lambda_{0}}}\mathcal{V}_{k_{\nu}}\times J_{m_{L(k_{\nu})},n_{L(k_{\nu})}}$ and we obtain
$$
1_{[\lambda_{0}-\delta,\lambda_{0}+\delta]}(H_{0})[H_{0},i\tilde{A}_{I}]1_{[\lambda_{0}-\delta,\lambda_{0}+\delta]}(H_{0})\geq \frac{1}{2}1_{[\lambda_{0}-\delta,\lambda_{0}+\delta]}(H_{0})\,,
$$
which is \eqref{eq:localMourre}\,.

\section{Comments and examples}
\label{sec:comEx}
The changes between the initial and the new versions requires several comments in particular to convince the reader that it does not change the applications which were made e.g. in \cite{GeNi2}. Some simple examples in dimension 2 are then given in order to illustrate the multistep construction of $\tilde{A}_{I}$\,.
\subsection{Comments}
\label{sec:comments}
\begin{description}
\item[a)] The explicit form of the conjugate operator $\tilde{A}_{I}$ is not important but it is important for the applications that it remains a first-order differential operator with $\mathcal{C}^{\infty}(M;\mathcal{L}(\mathcal{F}))$-coefficients. The main difference between ``Theorem''~3.1 in \cite{GeNi} and the new version Theorem~\ref{th:AI1I} is actually the local form \eqref{eq:localMourre} of Mourre's inequality around the energy $\lambda$\,. But this does not make any problem because Mourre theory  is essentially local in energy (see \cite{Mou1}\cite{Mou2}\cite{ABG}).
\item[b)] The local in energy version of Mourre inequality is actually the stable version  when perturbations $H_{0}+V$ are considered. From this point of view Theorem~3.3 of \cite{GeNi} which was a litteral application of the general Proposition~7.5.6 in \cite{ABG}, remains valid after replacing $A_{I}$ by its new version $\tilde{A}_{I}$\,. Note that a regularity $H_{0}\in C^{1+\varepsilon}(A)$ with $\varepsilon>0$\,, or $\varepsilon=0$ interpreted as a Dini-type continuity (see \cite{ABG}), is required. This was definitely incorrect in our initial version with $A=A_{I}$ and this works now with $A=\tilde{A}_{I}$\,.
\item[c)] In all the examples presented in \cite{GeNi} and in particular for the case of periodic Schr{\"o}dinger operators developed in \cite{GeNi2}, the Hilbert bundle $p_{\mathcal{F}}:\mathcal{F}\to M$ is actually a trivial one $\mathcal{F}=M\times \mathcal{H}'$\,, with $M=\tz^{d}=\rz^{d}/\zz^{d}$\,. The initial connection $\nabla^{M}$ can be taken as the trivial one and this is probably the reason why we were not very careful about the use of connections in \cite{GeNi}. However the now more accurate and correct version, shows that it is important to clarify this point especially when finite rank subbundles $p_{\pi_{m,n}\mathcal{F}}:\pi_{m,n}\mathcal{F}\to \omega_{m}$ are considered.
\item[d)] In \cite{GeNi2} with $\mathcal{F}=\tz^{d}\times \mathcal{H}'$ the perturbation $V$ that we considered were $\mathcal{L}(\mathcal{H}')$-valued pseudo-differential operator $V(D_{k})$ of order $-\mu$\,, $\mu>0$\,, with a scalar principal part. The negative order $-\mu<0$ implies that $V(D_{k})$ is a relatively compact perturbation of $H_{0}$ while the scalar principal part implies $V(D_{k})\in \mathcal{C}^{1+\mu}(\tilde{A}_{I})$ and $H_{0}+V(D_{k})\in \mathcal{C}^{1+\mu}(A_{U_{1},I2})$\,. All the arguments developed in \cite{GeNi2}, concerned with the absence of singular continuous spectrum, the virial theorem or propagation estimates, work after replacing $A_{I}$ by $\tilde{A}_{I}$\,. About the asymptotic completeness of the wave operators, which is a consequence of Proposition~7.5.6 in \cite{ABG} or Theorem~3.3 in \cite{GeNi} when $\mu>1$\,, we recall that it is still an open question in the long range case $\mu>0$\,, as it is discussed in \cite{GeNi2}.
\item[e)] About the absence of singular continuous spectrum for the free operator $H_{0}$ away from the set of thresholds $\tau$\,, a conjugate operator is not necessary and we refer the reader to \cite{Kuc} for a spectral measure argument for it.
\item[f)] Theorem~\ref{th:AI1I} and its proof were given in a global form, with $\tilde{A}_{I}$ constructed for the whole energy intervals $I\Subset \tilde{I}\Subset \rz\setminus\tau$\,. This was done in order to stick as much as possible to the initial version. This imposes the treatment of possibly several energy intervals $J_{m,n}$\,, $1\leq n\leq N_{m}$ above $\omega_{m}$\,. Actually the construction can be made simpler if one works from the beginning in $I=]\lambda-\delta,\lambda+\delta[$ and $\tilde{I}=]\lambda-2\delta,\lambda+2\delta[$ with $\delta>0$ small enough so that for every $m\in \left\{1,\ldots, M_{I}\right\}$\,, $N_{m}=1$\,. This simplifies the connections $\nabla^{\omega_{m},\pi}=\pi_{m,1}\nabla^{\omega_{m}}\pi_{m,1}$ on $p_{\pi_{m,1}\mathcal{F}}:\pi_{m,1}\mathcal{F}\to \omega_{m}$ with a single term and only one vector field $X_{m,1}$ above every $\omega_{m}$ is used in the construction of $\tilde{A}_{I}$\,. However while gluing the operators $\nabla^{\omega_{m},1}_{X_{m,1}}$ with the partition of unity $\sum_{m=1}^{M_{I}}g_{m}^{2}(k)\equiv 1$ in a neigborhood of $p_{M}(p_{\Sigma\to\rz}^{-1}([\lambda-\delta,\lambda+\delta])\cap \Sigma)$\,, the rank of the various projections $\pi_{m,1}$ changes and the  modification of the connections $\nabla^{m,1}$ into $\tilde{\nabla}^{m,1}$ is still necessary in order to ensure $H_{0}\in \mathcal{C}^{\infty}(\tilde{A}_{I})$\,. This is probably a simpler way for visualizing the whole construction, of which the generalization relies on the simple Lemma~\ref{le:Cstar}, and this is how it will be illustrated in the examples below.
\item[g)] Generically crossings of eigenvalues of a matricial self-adjoint operators occur along codimension 2 strata. When $\dim M=2$\,, singular strata with $\dim S< M$ are thus points and these singular points are removed while working in the energy intervals $I\Subset \rz\setminus \tau$\,. So when $\dim M=2$ and generically, $p_{\Sigma\to\rz}^{-1}(I)$ is made of disconnected strata of dimension $2$ locally diffeomorphic to $M$\,, and $H_{0}$ is locally (on $\Sigma$) a scalar operator. The analysis is much simpler. Such two dimensional problems were considered in \cite{FeWe} and appear in the modelling of graphene.
\item[h)] The assumption that the mapping $p_{\Sigma\to \rz}:\Sigma\to \rz$ is proper is used in a fundamental way, in order to make a stratification which is compatible with the projection $p_{\Sigma\to \rz}$\,. In some examples where a finite stratification is obviously given  it can be forgotten and the construction of a conjugate operator $\tilde{A}_{I}$  can be done without this assumption, especially when the differential of $p_{S\to\rz}$\,, the restriction of $p_{\Sigma\to \rz}$ to every stratum $S$\,, is uniformly away from $0$\,.
Then the construction relies on the local understanding of incidence of strata with the possible changes of the rank of the associated projectors. Another use  $p_{\Sigma\to\rz}:\Sigma\to \rz$ being proper, is for the construction of a self-adjoint operator for Nelson's commutator method, but this can be easy for some non compact manifold $M$ (e.g. $M=\rz^{d}$). Such a simple example is discussed below.
\item[i)] One may wonder about the mixing of the stratification of subanalytic sets, which makes sense in the analytic category with proper analytic mapping, with $\mathcal{C}^{\infty}$ partition of unity. It works for our construction of $\tilde{A}_{I}$\,. A question is whether the incidence of strata, as well as the associated spectral projectors, could be fully handled within the subanalytic category with the help of Lojasiewic inequalities. 
\end{description}

\subsection{Examples}
\label{sec:examples}

We give two  examples with $M=\rz^{2}$ and $\mathcal{H}'=\rz^{2}$ with pictures in order to visualize in simple cases our general construction and to make more explicit the problem set by the initial construction of $A_{I}$\,. Although these examples are non generic according to comment~\textbf{g)}, the two dimensional case, $M=\rz^{2}$\,, allows simple pictures.\\
{\red Other examples dimesion $3$ with  $M=\rz^{3}$ or $M=\tz^{3}$\,, motivated by the study of Maxwell equations were recently considered in \cite{Poi}. In this text a more accurately designed conjugate operator taking into account the more specific structure of the operator allows to consider the Limiti Absorption Principle at the threshold energies.}
\\

\medskip
\noindent\textbf{First example:} Consider
$$
H_{0}(k_{1},k_{2})=
\begin{pmatrix}
  k_{1}^{2}+k_{2}^{2}+k_{2}+k_{1}&k_{1}k_{2}\\
  k_{1}k_{2}& k_{1}^{2}+k_{2}^{2}+k_{2}-k_{1}
\end{pmatrix}
= k_{1}^{2}+k_{2}^{2}+k_{2}\mathrm{Id}_{\rz^{2}}+
k_{1}\begin{pmatrix}
  1&k_{2}\\
  k_{2}&-1
\end{pmatrix}\quad (k_{1},k_{2})\in \rz^{2}
$$
The characteristic variety is given by
$$
\Sigma=\left\{k_{1}^{2}+k_{2}^{2}+k_{2}\pm k_{1}\sqrt{1+k_{2}^{2}}, (k_{1},k_{2})\in \rz^{2}\right\}
$$
and is partitionned into $15$ strata given by
$$
\lambda=k_{1}^{2}+k_{2}^{2}+k_{2}\pm k_{1}\sqrt{1+k_{2}^{2}}
$$
with the following choices
\begin{itemize}
\item $\pm k_{1}>0$\,, $\pm (k_{2}+1/2)>0$\,, $\dim S=2$\,, multiplicity $1$\,, in blue or red on Figure~\ref{fig:1};
\item $k_{1}=0$\,, $\pm (k_{2}+1/2)>0$\,, $\dim S=1$\,, multiplicity 2, green on Figure~\ref{fig:1};
\item $\pm k_{1}>0$\,, $k_{2}=-1/2$\,, $\dim S=1$\,, multiplicity 1, black on Figure~\ref{fig:1};
\item $k_{1}=0$, $k_{2}=-1/2$\,, $\dim S=0$\,, multiplicity 2, black on Figure~\ref{fig:1}.
\end{itemize}
\begin{figure}[!h]
  \centering
\begin{tikzpicture}
  \pgftext{\includegraphics[scale=0.5]{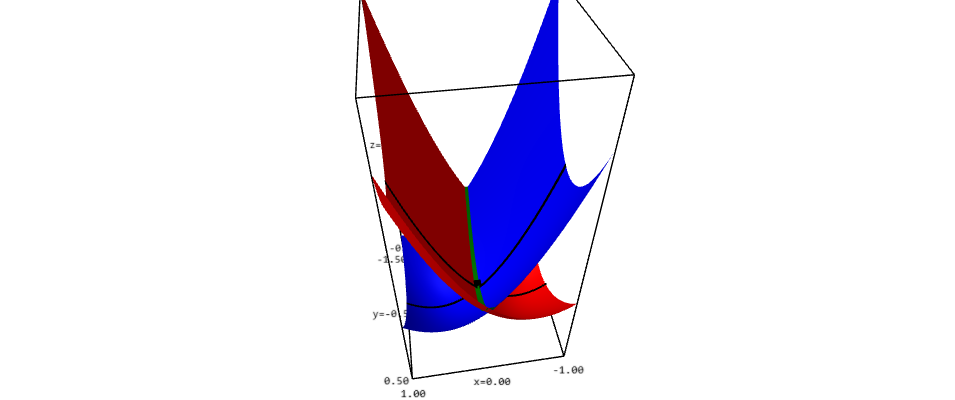}} at (0,0)
\end{tikzpicture}
\caption{Stratification of $\Sigma$}
  \label{fig:1}
\end{figure}

There is one threshold at energy $-\frac{1}{4}$ corresponding to $(k_{1}=0,k_{2}=-1/2,\lambda=-\frac{1}{4})$\,.
While working in the energy set $\tilde{I}\Subset ]-\frac{1}{4},+\infty[$ the number of strata can be reduced to $5$ the red, blue ones ($\dim S=2$, multiplicity $1$) and the green one ($\dim S=1$\,, multiplicity 2)  on Figure~\ref{fig:1}.\\
The situation is better seen by considering the level sets in $\rz^{2}$ around an energy $\lambda_{0}$ (in Figure~\ref{fig:2} $\lambda_{0}=1$)\,.
\begin{figure}[!h]
  \centering
  \begin{tikzpicture}
    \pgftext{\includegraphics[scale=0.7]{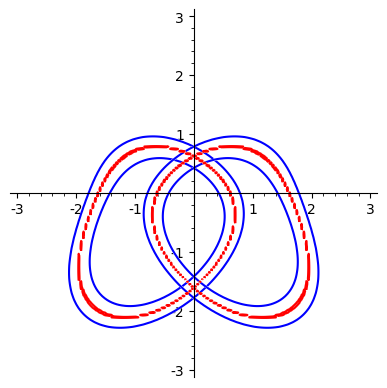}} at (0,0);
     \draw [very thick,->] (1,3)--(0.1,0.75);
     \draw [very thick,->] (1,3)--(0.1,-1.7);
     \node (P2) at (3,3.1){$\dim S=1$, multiplicity 2};
     \draw [very thick,->] (3,-1.3)--(1.5,0.5);
     \draw [very thick,->] (3,-1.3)--(0.6,-0.1);
     \draw [very thick,->] (3,-1.3)--(-0.7,-0.5);
     \draw [very thick,->] (3,-1.3)--(-2,-0.9);
     \node (P2) at (5,-1.7){$\dim S=2$, multiplicity 1};
  \end{tikzpicture}
  \caption{Level sets $\lambda\in[\lambda_{0}-\delta,\lambda_{0}+\delta]$\,, $\lambda_{0}=1$\,, $\delta=0.4$ (blue), $\delta=0.04$ (red).}
  \label{fig:2}
\end{figure}
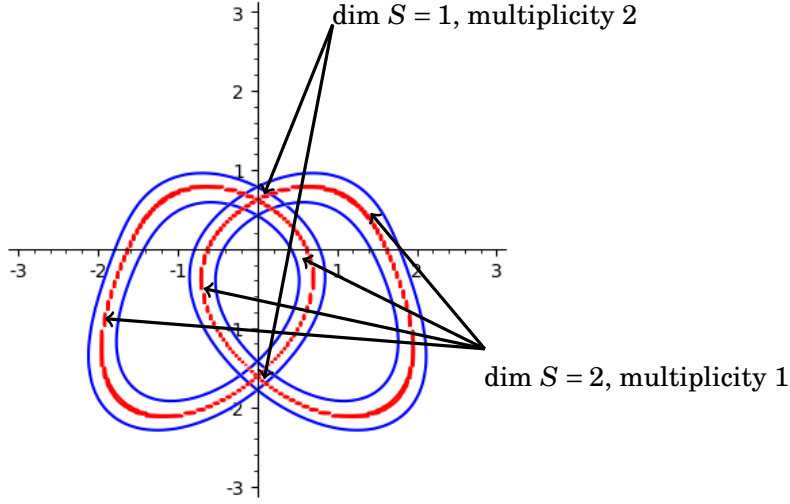
Let us look more carefully around a point along the $k_{2}$-axis where $\lambda_{0}$ is an eigenvalue with multiplicity $2$\,. Three points $k_{m_{0}}$\,, $\dim S_{k_{m_{0}}}=1$\,, $k_{m_{1}}$ and $k_{m_{2}}$ with $\dim S_{k_{m_{1}}}=\dim S_{k_{m_{2}}}=2$\,, the boundaries of open set $\omega_{m_{j}}$ (dashed lines) and of $\supp g_{m_{j}}$ (black lines) are represented for $j=0,1,2$ in Figure~\ref{fig:3}.
\begin{figure}[!h]
  \centering
  \begin{tikzpicture}
    \pgftext{\includegraphics[scale=0.7]{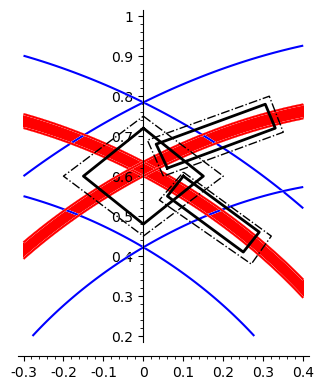}} at (0,0);
    \node (P0) at (-0.2,0.37) {$k_{m_{0}}$};
    \node (P1) at (1.2,1.2) {$k_{m_{1}}$};
    \node (P2) at (1.2,-0.4) {$k_{m_{2}}$};
  \end{tikzpicture}
  \caption{Open sets $\omega_{m}$ (dashed lines) and $\supp g_{m}$ (black lines) around the points $k_{m}$\,.\\ Level sets $\lambda\in [\lambda_{0}-\delta,\lambda_{0}+\delta]$\,, $\lambda_{0}=1$\,, $\delta=0.4$ (blue), $\delta=0.04$ (red).}
  \label{fig:3}
\end{figure}

In Figure~\ref{fig:3} the projectors $\pi_{m_{j},1}$ have the minimal rank $1$ for $j=1,2$ and the connection $\nabla^{m_{j},1}$ need not to be modified above $\omega_{m_{j}}$\,. It is not the case for $m_{0}$  where $\rank \pi_{m_{0},1}=2$ while $\supp g_{m_{0}}$ intersect $\supp g_{m_{1}}$ and $\supp g_{m_{2}}$\,. The connection $\tilde{\nabla}^{m_{0},1}$ differs from the trivial connection $\nabla^{m_{0},1}$\,, $\nabla_{X}^{m_{0},1}=X\otimes \mathrm{Id}_{\rz^{2}}$\,.  Above a neighborhood of $\supp g_{m_{j}}\cap \supp g_{m_{0}}$ for $j=1,2$ (black areas in Figure~\ref{fig:4}), it equals $\tilde{\nabla}^{m_{0},1}=\pi_{m_{j},1}\nabla^{m_{0},1}\pi_{m_{j},1}+(\mathrm{Id}_{\rz^{2}}-\pi_{m_{j},1})\nabla^{m_{0},1}(\mathrm{Id}_{\rz^{2}}-\pi_{m_{j},1})$\,.
\begin{figure}[!h]
  \centering
  \begin{tikzpicture}
    \pgftext{\includegraphics[scale=0.5]{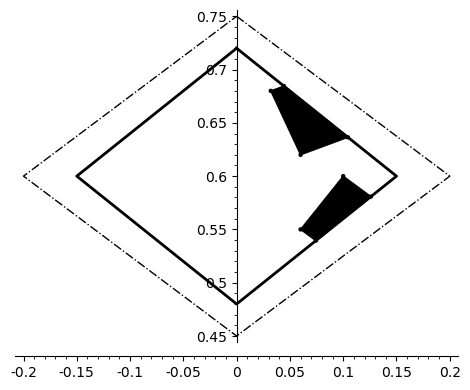}} at (0,0);
    \node (P) at (0.4,0.5) {$k_{m_{0}}$};
  \end{tikzpicture}
  \caption{Localization in $\omega_{m_{0}}$ (dashed limits) of $\supp g_{m_{0}}\cap \supp g_{m_{j}}$ (in black)\,, $j=1,2$\,, around which the connection must be modified.}
  \label{fig:4}
\end{figure}

\pagebreak[4]
\noindent\textbf{Second example:} This is a simplified version of the first one, for which we will explicitely compare the two operators $A_{I}$ and $\tilde{A}_{I}$\,, with
$$
H_{0}(k_{1},k_{2})=
\begin{pmatrix}
  k_{2}+k_{1}& k_{1}k_{2}\\
  k_{1}k_{2}&k_{2}-k_{1}
\end{pmatrix}
=k_{2}\mathrm{Id}_{\rz^{2}}+k_{1}
\begin{pmatrix}
  1&k_{2}\\
  k_{2}&-1
\end{pmatrix}\,.
$$
Here the characteristic variety is
$$
\Sigma=\left\{k_{2}\pm k_{1}\sqrt{1+k_{2}^{2}}\,, (k_{1},k_{2})\in \rz^{2}\right\}\,.
$$
It does not fullfil the condition that $p_{\Sigma\to\rz}:\Sigma\to \rz$ is proper. But as mentionned in Comment~\textbf{h)} above, it is not a problem because there is an obvious finite stratification of $\Sigma$ compatible with $p_{\Sigma\to \rz}:\Sigma\to \rz$\,.
It is given by the 5 strata given by
$$
\lambda=\lambda_{\pm}(k_{1},k_{2})=k_{2}\pm k_{1}\sqrt{1+k_{2}^{2}}
$$
with the following choices:
\begin{itemize}
\item $\pm k_{1}>0$\,, $\dim S_{\pm,\pm}=2$\,, multiplicity 1, in blue or red on Figure~\ref{fig:5};
\item $k_{1}=0$\,, $\dim S_{0}=1$\,, multiplicity 2, in green on Figure~\ref{fig:5}.
\end{itemize}
\begin{figure}[!h]
  \centering
\begin{tikzpicture}
  \pgftext{\includegraphics[scale=0.5]{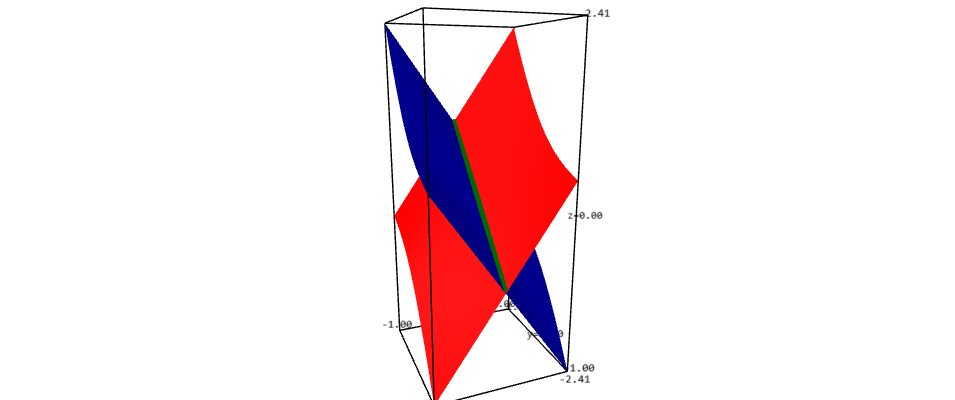}} at (0,0)
\end{tikzpicture}
\caption{Stratification of $\Sigma$ in the second example.}
  \label{fig:5}
\end{figure}
There is no threshold because $\partial_{k_{2}}k_{2}=1$ along $p_{M}S_{0}=\left\{(0,k_{2}), k_{2}\in \rz\right\}$ while $\nabla \lambda_{\pm}=
\begin{pmatrix}
  \pm\sqrt{1+k_{2}^{2}}\\
  1\pm \frac{k_{2}}{\sqrt{1+k_{2}^{2}}}k_{1}
\end{pmatrix}
$ satisfies $|\nabla\lambda_{\pm}|\geq 1$ on $p_{M}(S_{\pm,\pm})=\left\{(k_{1},k_{2})\,, \pm k_{1}>0\right\}$\,. In particular $\frac{\nabla \lambda_{\pm}}{|\nabla \lambda_{\pm}|^{2}}.\nabla$ is a smooth vector field on $p_{M}(S_{\pm,\pm})$ such that $\frac{\nabla\lambda_{\pm}}{|\nabla \lambda_{\pm}|^{2}}.\nabla \lambda_{\pm}=1$\,.
Notice also that
$$
\partial_{k_{2}}H_{0}=
\begin{pmatrix}
  1&k_{1}\\
  k_{1}&1
\end{pmatrix}
\geq \frac{1}{2}\mathrm{Id}_{\rz^{2}}\quad\text{for}~|k_{1}|\leq \frac{1}{2}\,.
$$
The spectral projectors associated with the eigenvalues $\lambda_{\pm}(k_{1},k_{2})$ for $k_{1}\neq 0$ are given by
$$
\pi_{+}=
\frac{1}{2\sqrt{1+k_{2}^{2}}}\begin{pmatrix}
 1+\sqrt{1+k_{2}^{2}}& k_{2}\\
 k_{2}&-1+\sqrt{1+k_{2}}                              
\end{pmatrix}\quad,\quad \pi_{-}=
\frac{1}{2\sqrt{1+k_{2}^{2}}}
\begin{pmatrix}
  \sqrt{1+k_{2}^{2}}-1&-k_{2}\\
  -k_{2}&\sqrt{1+k_{2}^{2}}+1
\end{pmatrix}\,.
$$
Take $g_{0}\in \mathcal{C}^{\infty}_{comp}(]-1/2,1/2[;\rz)$\,, $g_{-}\in \mathcal{C}^{\infty}(]-\infty,\frac{1}{4}[;\rz)$ and $g_{+}\in \mathcal{C}^{\infty}(]\frac{1}{4},+\infty[;\rz)$ such that $g_{0}^{2}+g_{-}^{2}+g_{+}^{2}\equiv 1$ on $\rz$\,. The operator $A_{I}$ of \cite{GeNi} is then given by
$$
-iA_{I}=g_{0}(k_{1})\frac{\partial}{\partial k_{2}}g_{0}(k_{1})
+\sum_{\varepsilon_{1},\varepsilon_{2}\in \left\{+,-\right\}}g_{\varepsilon_{1}}(k_{1})\pi_{\varepsilon_{2}}\left[\frac{\nabla \lambda_{\varepsilon_{2}}}{|\nabla \lambda_{\varepsilon_{2}}|^{2}}.\nabla+\frac{1}{2}\mathrm{div}\,\left(\frac{\nabla \lambda_{\varepsilon_{2}}}{|\nabla \lambda_{\varepsilon_{2}}|^{2}}\right)\right]\pi_{\varepsilon_{2}}g_{\varepsilon_{1}}(k_{1})\,.
$$
Then $[H_{0},iA_{I}]$ equals
$$
g_{0}^{2}(k_{1})
\begin{pmatrix}
  1&k_{1}\\
  k_{1}&1
\end{pmatrix}
+g_{+}^{2}(k_{1})+g_{-}^{2}(k_{1})=\mathrm{Id}_{\rz^{2}}+g_{0}^{2}(k_{1})k_{1}
\begin{pmatrix}
  0&1\\
  1&0
\end{pmatrix}
$$
Because $\left[
  \begin{pmatrix}
    0&1\\
    1&0
  \end{pmatrix}
  \,,\, \pi_{\pm}\right]\neq 0$\,, the double commutator $[[H_{0},iA_{I}],iA_{I}]$ is a first-order differential operator with a non vanishing principal part for $k_{1}\in \supp g_{0}\cap \supp g_{\pm}$ and it is not bounded.\\
For $\tilde{A}_{I}$ we change the trivial connection in $\left\{(k_{1}\,,\,k_{2}),\; |k_{1}|<1/2\right\}$: Take $\theta_{0}\in \mathcal{C}^{\infty}_{comp}(]-1/4,1/4[;\rz)$ and $\theta_{1}\in \mathcal{C}^{\infty}(\rz;\rz)$\,, $\supp \theta_{1}\subset \rz\setminus[-\frac{1}{8},\frac{1}{8}]$ and $\theta_{0}^{2}+\theta_{1}^{2}\equiv 1$ on $\rz$\,.
Then replace $\frac{\partial}{\partial k_{2}}$ by
$$
\tilde{\nabla}_{\frac{\partial}{\partial k_{2}}}=\theta_{0}(k_{1})\frac{\partial}{\partial k_{2}}\theta_{0}(k_{1})
+\theta_{1}(k_{1})[\pi_{+}\frac{\partial}{\partial k_{2}}\pi_{+}+\pi_{-}\frac{\partial}{\partial k_{2}}\pi_{-}]\theta_{1}(k_{1})
$$
and take
$$
-i\tilde{A}_{I}=g_{0}(k_{1})\tilde{\nabla}_{\frac{\partial}{\partial k_{2}}}g_{0}(k_{1})
+\sum_{\varepsilon_{1},\varepsilon_{2}\in \left\{+,-\right\}}g_{\varepsilon_{1}}(k_{1})\pi_{\varepsilon_{2}}\left[\frac{\nabla \lambda_{\varepsilon_{2}}}{|\nabla \lambda_{\varepsilon_{2}}|^{2}}.\nabla+\frac{1}{2}\mathrm{div}\,\left(\frac{\nabla \lambda_{\varepsilon_{2}}}{|\nabla \lambda_{\varepsilon_{2}}|^{2}}\right)\right]\pi_{\varepsilon_{2}}g_{\varepsilon_{1}}(k_{1})\,.
$$
Computing $\tilde{\nabla}_{\frac{\partial}{\partial k_{2}}}H_{0}$ gives
\begin{eqnarray*}
  \tilde{\nabla}_{\frac{\partial}{\partial k_{2}}}H_{0}
  &=&\theta_{0}^{2}(k_{1})
\begin{pmatrix}
  1&k_{1}\\
  k_{1}& 1
\end{pmatrix}
+\theta_{1}^{2}(k_{1})\left[\frac{\partial \lambda_{+}}{\partial k_{2}}\pi_{+}+\frac{\partial \lambda_{-}}{\partial k_{2}}\pi_{-}\right]
  \\
  &=&\mathrm{Id}_{\rz^{2}}+\theta_{0}^{2}(k_{1})k_{1}
\begin{pmatrix}
  0&1\\
  1&0
\end{pmatrix}
+\theta_{1}^{2}(k_{1})\frac{k_{2}k_{1}}{\sqrt{1+k_{2}^{2}}}(\pi_{+}-\pi_{-})
\end{eqnarray*}
and the commutator $[H_{0},i\tilde{A}_{I}]$ equals now
$$
\mathrm{Id}_{\rz^{2}}+k_{1}\theta_{0}^{2}(k_{1})g_{0}^{2}(k_{1})
\begin{pmatrix}
  0&1\\
  1&0
\end{pmatrix}
+\theta_{1}^{2}(k_{\red 1})g_{0}^{2}(k_{1})\frac{k_{2}k_{1}}{\sqrt{1+k_{2}^{2}}}(\pi_{+}-\pi_{-})\,.
$$
Because $\theta_{0}g_{0}\equiv 0$ on $\supp g_{\pm}$ and $\pi_{+}-\pi_{-}$ commutes with $\pi_{\pm}$\,, the double commutator $[[H_{0},i\tilde{A}_{I}],i\tilde{A}_{I}]$ is now a zeroth order operator. It is bounded and the iterated commutator can be checked to be bounded{\red: Actually all the non diagonal terms with respect to
the decomposition $\mathrm{Id}_{\rz^{2}}=\pi_{+}+\pi_{-}$\,, generated after every commutation with $i\tilde{A}_{I}$\,, are supported on $\supp \theta_{0}$ and therefore vanish on $\supp g_{\pm}$\,.}\\
Here the Mourre estimate can be written globally because
$$
\theta_{1}^{2}(k_{1})g_{0}^{2}(k_{1})\frac{k_{2}k_{1}}{\sqrt{1+k_{2}^{2}}}(\pi_{+}-\pi_{-})\geq -\frac{1}{2}\theta_{1}^{2}(k_{1})\mathrm{Id}_{\rz^{2}}
$$
and
$$
k_{1}g_{0}^{2}(k_{1})\theta_{0}^{2}(k_{1})
\begin{pmatrix}
  1&0\\
  0&1
\end{pmatrix}\geq -\frac{1}{2}\theta_{0}^{2}(k_{1})\mathrm{Id}_{\rz^{2}}\,.
$$
We obtain
$$
\left[H_{0}\,,\, i\tilde{A}_{I}\right]\geq \frac{1}{2}\mathrm{Id}_{\rz^{2}}\,.
$$
\pagebreak[4]

\noindent{\red \textbf{Acknowledgements:} The authors warmly thank Olivier Poisson for pointing out the problem which is fixed in this text.
}

\bibliographystyle{plain}

\end{document}